\documentclass[amstex,12pt]{amsart}
\usepackage{geometry}
\usepackage{mathtools,amssymb,amsthm,mathrsfs,color,lineno,paralist,graphicx,float}
\usepackage[linktocpage,backref=page,colorlinks,
linkcolor=red,
anchorcolor=red,
citecolor=red,
]{hyperref}
\usepackage{color}
\usepackage[T1]{fontenc}
\usepackage[utf8]{inputenc}
\usepackage{mathrsfs}
\usepackage{cleveref}
\usepackage{mathscinet}
\usepackage{calc}
\usepackage{etoolbox}

\newtheorem{thm}{Theorem}[section]
\theoremstyle{plain}
\theoremstyle{definition}
\newtheorem{proposition}[thm]{Proposition}

\newtheorem{lemma}[thm]{Lemma}
\newtheorem{definition}[thm]{Definition}
\newtheorem{claim}[thm]{Claim}

\newtheorem{rem}[thm]{Remark}

\numberwithin{equation}{section}

\newcommand{\uo}{\omega}

\definecolor{lgray}{gray}{0.9}

\def\leq{\leqslant}
\def\geq{\geqslant}
\def\tilde{\widetilde}
\def\hat{\widehat}

\allowdisplaybreaks
\renewcommand*{\backref}[1]{}
\renewcommand*{\backrefalt}[4]{\quad \tiny
  \ifcase #1 (\textbf{NOT CITED.})
  \or    (Cited on page~#2.)
  \else   (Cited on page~#2.)
  \fi}
\makeatletter

\title[Full-horseshoes for GSNS]{Full-horseshoes for  Galerkin truncations of   the 2D Navier-Stokes equation with  degenerate stochastic forcing}

\author{Wen Huang}
\address[Wen Huang]{School of Mathematical Sciences\\ University of Science and Technology of China\\ Hefei, Anhui, 230026, China}
\email{wenh@mail.ustc.edu.cn}

\author{Jianhua Zhang}
\address[Jianhua Zhang]{School of Mathematical Sciences\\ University of Science and Technology of China\\ Hefei, Anhui, 230026, China}
\email{leapforg@mail.ustc.edu.cn}

\keywords{Stochastic flow; stationary measure; entropy; full-horseshoes; Lyapunov exponents}
\date{\today}
\subjclass[2020]{37H05, 37A50, 60H10.}

\begin{document}

\begin{abstract}
In this paper, we mainly study the chaotic phenomenon   of  Galerkin truncations of   the 2D Navier-Stokes equation under degenerate stochastic forcing  and large-scale.  We use a kind of chaotic structure named full-horseshoes to describe it. It is proved that  if the  stochastic forcing satisfies some kind of hypoelliptic condition, then  the system  has full-horseshoes.  

 \end{abstract}
\maketitle


\section{Introduction}
\label{23-2-12-1355}
 Turbulent dynamical systems are ubiquitous complex systems in hydrodynamics, such as \cite{majda2006nonlinear,pope2000turbulent,salmon1998lectures,vallis2017atmospheric},  and are characterized by a large dimensional phase space and a large dimension of unstable directions in phase space.
 One of the main goals in the development of the theory of chaotic dynamical systems has been to understand the chaotic phenomenon of turbulent dynamical systems.  In this paper, we focus on  Galerkin truncations of  the 2D stochastic Navier-Stokes equation on tours  (being abbreviated as GSNS). This model was  initiated  by E and Mattingly in \cite{MR1846802}. And they proved the unique ergodicity  of  stationary measure of  GSNS under  a kind of degenerate stochastic forcing.  Later,  Hairer and Mattingly  showed  the unique ergodicity  of  stationary measure  of  2D  stochastic Navier-Stokes equation (being abbreviated as SNS) under more general  degenerate  stochastic forcing in \cite{MR2259251}.
 
  However, there are few mathematically rigorous  results to describe the  chaotic phenomenon  of GSNS or  SNS, even other turbulent dynamical systems. In the recent,  Bedrossian, Blumenthal  and Punshon-Smith made a breakthrough and  firstly proved  a  kind of  turbulent dynamical systems under the H\"omander's  parabolic bracket spanning assumption have positive Lyapunov exponents  with respect to its unique stationary measure  in \cite{MR4372219}.   Particularly, GSNS has positive Lyapunov exponents  with respect to its unique stationary measure \cite{bedrossian2021chaos}.  It is well-known that positive Lyapunov exponent implies
  the sensitive dependence on initial conditions. Then, it is natural to ask  whether there is  a kind of  chaotic  structure   in GSNS, such as Smale horseshoes \cite{MR0182020} or others.

    Full-horseshoe was firstly introduced by the first author and Lu in \cite{HL} to investigate the complex behaviors of infinite-dimensional random dynamical systems.  It imitates the process of coin toss  and   is  a weaker chaotic structure than Smale horseshoe.   The purpose of this paper is to characterize  chaotic phenomenon  of  the GSNS  with full-horseshoes. To the authors' knowledge, this is the  first result about  chaotic structure for the large dimensional turbulent dynamical systems.  Actually, we  obtain  a  sufficient condition to guarantee the existence of full-horseshoes for general systems (see \Cref{23-3-13-2003}). 

Let's begin with the  incompressible  2D  stochastic Navier-Stokes equation in vorticity form on tours $\mathbb{T}^2:=\mathbb{R}^2/2\pi\mathbb{Z}^2$ as 
\begin{align}
\label{22-12-05-03}
\partial_tq=\epsilon\Delta q-\boldsymbol{p}\cdot \nabla q+\partial_t W,
\end{align}
where $\boldsymbol{p}:=\nabla^{\perp}(-\Delta)^{-1}q$  is the divergence free velocity field, 
$\epsilon\in(0,1)$ is the constant viscosity,  and  $\partial_t W$ is the  stochastic  force  described as follows.

Let 
$$\mathbb{Z}^2_+=\{(k_1, k_2)\in \mathbb{Z}^2: k_2>0\}\cup\{(k_1, k_2)\in \mathbb{Z}^2: k_1>0, k_2=0\}\text{ and } \mathbb{Z}^2_0:=\mathbb{Z}^2_+\cup-\mathbb{Z}^2_+,$$ and let $(\Omega,\mathscr{F},\mathbb{P})$  be  an infinite-dimensional  Wiener space, i.e. 
\begin{align}
\label{23-8-28-1017}
(\Omega,\mathscr{F},\mathbb{P})=\big(C_0([0,+\infty), \mathbb{R}),\mathscr{F}_0,\mathbb{P}_0\big)^{\mathbb{Z}^2_0},
\end{align}
where $\big(C_0([0,+\infty), \mathbb{R}),\mathscr{F}_0,\mathbb{P}_0\big)$ is  the standard one-dimensional  Wiener space.  Then 
$$\{W^{(\boldsymbol{k},1)}(\omega):=\omega_{\boldsymbol{k}},\medspace W^{(\boldsymbol{k},2)}:=\omega_{-\boldsymbol{k}}\}_{\boldsymbol{k}\in\mathbb{Z}^2_+}$$  is a    family  of independent   one-dimensional Wiener processes on $(\Omega,\mathscr{F},\mathbb{P})$.  Throughout  this paper, we will consider a white-in-time stochastic forcing $\partial_tW$ being the form   
$$\partial_tW=\sum_{\boldsymbol{k}\in\mathbb{Z}_+^2}e_{(\boldsymbol{k},1)}\cos(\boldsymbol{k}\cdot\boldsymbol{x})\dot{W}^{(\boldsymbol{k},1)}+e_{(\boldsymbol{k},2)}\sin(\boldsymbol{k}\cdot\boldsymbol{x})\dot{W}^{(\boldsymbol{k},2)}, $$
where  $e_{(\boldsymbol{k},1)}$ and $e_{(\boldsymbol{k},2)}$ are constants which satisfy that 
$$e_{(\boldsymbol{k},1)}e_{(\boldsymbol{k},2)}=0  \text{ if and only if }  e_{(\boldsymbol{k},1)}=e_{(\boldsymbol{k},2)}=0.$$

 Let $q(t,\boldsymbol{x})=\sum_{\boldsymbol{k}\in\mathbb{Z}^2_{+}}q_{(\boldsymbol{k},1)}(t)\cos(\boldsymbol{k}\cdot\boldsymbol{x}) +q_{(\boldsymbol{k},2)}(t)\sin(\boldsymbol{k}\cdot\boldsymbol{x}) $. Using Fourier transformation, write \Cref{22-12-05-03} as an infinite system of  stochastic ordinary differential equation on $\mathbb{Z}^2_+$,
 \begin{align}
\label{23-10-8-1938-1}  \dot{q}_{(\boldsymbol{k},1)}=&\frac{1}{2}\sum_{\substack{\boldsymbol{i}+\boldsymbol{j}=\boldsymbol{k}\\ \boldsymbol{i},\boldsymbol{j}\in\mathbb{Z}^2_+}}c_{\boldsymbol{i}\boldsymbol{j}}(q_{(\boldsymbol{i},1)}q_{(\boldsymbol{j},1)}-q_{(\boldsymbol{i},2)}q_{(\boldsymbol{j},2)}) -\sum_{\substack{\boldsymbol{i}-\boldsymbol{j}=\boldsymbol{k}\\ \boldsymbol{i},\boldsymbol{j}\in\mathbb{Z}^2_+}}c_{\boldsymbol{i}\boldsymbol{j}}(q_{(\boldsymbol{i},1)}q_{(\boldsymbol{j},1)}+q_{(\boldsymbol{i},2)}q_{(\boldsymbol{j},2)}) 
  \\\notag\qquad&-\epsilon |\boldsymbol{k}|^2q_{(\boldsymbol{k},1)}+\frac{e_{(\boldsymbol{k},1)}}{2}\dot{W}^{(\boldsymbol{k},1)}
\\\label{23-10-8-1938-2} \dot{q}_{(\boldsymbol{k},2)}=&\frac{1}{2}\sum_{\substack{\boldsymbol{i}+\boldsymbol{j}=\boldsymbol{k}\\ \boldsymbol{i},\boldsymbol{j}\in\mathbb{Z}^2_+}}c_{\boldsymbol{i}\boldsymbol{j}}(q_{(\boldsymbol{i},1)}q_{(\boldsymbol{j},2)}+q_{(\boldsymbol{i},2)}q_{(\boldsymbol{j},1)}) -\sum_{\substack{\boldsymbol{i}-\boldsymbol{j}=\boldsymbol{k}\\ \boldsymbol{i},\boldsymbol{j}\in\mathbb{Z}^2_+}}c_{\boldsymbol{i}\boldsymbol{j}}(q_{(\boldsymbol{i},2)}q_{(\boldsymbol{j},1)}-q_{(\boldsymbol{i},1)}q_{(\boldsymbol{j},2)})
\\\notag\qquad&-\epsilon |\boldsymbol{k}|^2q_{(\boldsymbol{k},2)}-\frac{e_{(\boldsymbol{k},2)}}{2}\dot{W}^{(\boldsymbol{k},2)},
\end{align}
where $c_{\boldsymbol{i},\boldsymbol{j}}:=\langle\boldsymbol{i}^{\perp},\boldsymbol{j}\rangle(\frac{1}{|\boldsymbol{j}|^2}-\frac{1}{|\boldsymbol{i}|^2})\text{ with }\boldsymbol{i}^{\perp}=(i_2,-i_1).$

 In numerical simulations,  there is a well-known method  as Galerkin truncation to approximate PDEs.  Under our setting,   this is  to  fix a  positive integer $N$ and   restrict the indices of  \eqref{23-10-8-1938-1} and \eqref{23-10-8-1938-2}   on  following truncated lattice 
\begin{align}
\label{23-3-6-1456}
\mathbb{Z}^2_{+,N}=\{\boldsymbol{k}\in\mathbb{Z}_+^2: |\boldsymbol{k}|_{\infty}\leq N\}, \quad\text{where } |\boldsymbol{k}|_{\infty}:=\max\{|k_1|, |k_2|\}.
\end{align}
Then  we obtain a  stochastic differential equation  on $\mathbb{Z}^2_{+,N}$ as 
\begin{align}
\label{23-10-8-1953-1}\tag{$1.3_N$}  \dot{q}_{(\boldsymbol{k},1)}=&\frac{1}{2}\sum_{\substack{\boldsymbol{i}+\boldsymbol{j}=\boldsymbol{k}\\ \boldsymbol{i},\boldsymbol{j}\in\mathbb{Z}^2_{+,N}}}c_{\boldsymbol{i}\boldsymbol{j}}(q_{(\boldsymbol{i},1)}q_{(\boldsymbol{j},1)}-q_{(\boldsymbol{i},2)}q_{(\boldsymbol{j},2)}) -\sum_{\substack{\boldsymbol{i}-\boldsymbol{j}=\boldsymbol{k}\\ \boldsymbol{i},\boldsymbol{j}\in\mathbb{Z}^2_{+,N}}}c_{\boldsymbol{i}\boldsymbol{j}}(q_{(\boldsymbol{i},1)}q_{(\boldsymbol{j},1)}+q_{(\boldsymbol{i},2)}q_{(\boldsymbol{j},2)}) 
  \\\notag\qquad&-\epsilon |\boldsymbol{k}|^2q_{(\boldsymbol{k},1)}+\frac{e_{(\boldsymbol{k},1)}}{2}\dot{W}^{(\boldsymbol{k},1)}
\\\label{23-10-8-1953-2}\tag{$1.4_N$}\dot{q}_{(\boldsymbol{k},2)}=&\frac{1}{2}\sum_{\substack{\boldsymbol{i}+\boldsymbol{j}=\boldsymbol{k}\\ \boldsymbol{i},\boldsymbol{j}\in\mathbb{Z}^2_{+,N}}}c_{\boldsymbol{i}\boldsymbol{j}}(q_{(\boldsymbol{i},1)}q_{(\boldsymbol{j},2)}+q_{(\boldsymbol{i},2)}q_{(\boldsymbol{j},1)}) -\sum_{\substack{\boldsymbol{i}-\boldsymbol{j}=\boldsymbol{k}\\ \boldsymbol{i},\boldsymbol{j}\in\mathbb{Z}^2_{+,N}}}c_{\boldsymbol{i}\boldsymbol{j}}(q_{(\boldsymbol{i},2)}q_{(\boldsymbol{j},1)}-q_{(\boldsymbol{i},1)}q_{(\boldsymbol{j},2)})
\\\notag\qquad&-\epsilon |\boldsymbol{k}|^2q_{(\boldsymbol{k},2)}-\frac{e_{(\boldsymbol{k},2)}}{2}\dot{W}^{(\boldsymbol{k},2)}.
\end{align}
 It is clear that \eqref{23-10-8-1953-1}-\eqref{23-10-8-1953-2} defines a  stochastic flow of $C^{\infty}$ diffeomorphisms (for example, see \cite{MR4372219,MR4408018})   as 
\begin{align}
\label{23-2-18-1254}
\Phi: [0,+\infty)\times\Omega\times \mathbb{R}^{d}\to\mathbb{R}^d, \quad (t,\omega,x)\mapsto\Phi^t_{\omega}(x),
\end{align}
where $d:=4N(N+1)$, with following properties
\begin{enumerate}[(i)]
\item for $\mathbb{P}$-a.s. $\omega\in\Omega$, the mapping $t\mapsto \Phi^t_{\omega}$ is continuous from $[0,+\infty)$ to $\text{Diff}^{\infty}(\mathbb{R}^d)$ endowed with relative compact open topology;
\item for $\mathbb{P}$-a.s. $\omega\in\Omega$, one has that  $\Phi^t_{\theta^s\omega}\circ\Phi^s_{\omega}=\Phi_{\omega}^{t+s}$ for any $t,s\geq0$,  and $\Phi^0_{\omega}=\text{Id}_{\mathbb{R}^d}$, where $\theta^t: \Omega\to\Omega$ is  Wiener shift defined as $\theta^t(\omega)=\omega(\cdot+t)-\omega(t)$;
\item  for any $0\leq t_0<t_1<\cdots <t_n$,  the mappings $\omega\mapsto \Phi^{t_n-t_{n-1}}_{\theta^{t_{n-1}}\omega}, \dots, \omega\mapsto \Phi^{t_1-t_0}_{\theta^{t_0}\omega}$ are indenpendent random variables from  $\Omega\to \text{Diff}^{\infty}(\mathbb{R}^d)$. And the distribution of  $\omega\mapsto \Phi^{t}_{\theta^s\omega}$ only depends on $t$.
\end{enumerate}
To ensure that  the stochastic flow of  \eqref{23-10-8-1953-1}-\eqref{23-10-8-1953-2}  has a unique stationary measure for all $\epsilon\in(0,1)$,  we  review a hypoelliptic condition (see \cite{MR4372219} for  details) for  driven model 
\begin{align}
\label{23-10-9-1603}
\mathcal{K}_N:=\{\boldsymbol{k}\in\mathbb{Z}^2_{+,N}: e_{(\boldsymbol{k},1)}e_{(\boldsymbol{k},2)}>0  \},
\end{align}
where $N$ is a positive integer and $\mathbb{Z}^2_{+,N}$ is defined as \eqref{23-3-6-1456}.
\begin{definition}
$\mathcal{K}_N$ is  hypoelliptic if $\mathbb{Z}^2_{0,N}=\bigcup_{n=0}^{+\infty}\mathcal{Z}^n$,  where
\begin{align*}
&\mathbb{Z}^2_{0,N}:=\{\boldsymbol{k}\in\mathbb{Z}^2: 0<|\boldsymbol{k}|_{\infty}\leq N\},
\\& \mathcal{Z}^0=\mathcal{K}_N\cup(-\mathcal{K}_N),
\\&\mathcal{Z}^n=\{\boldsymbol{k}\in\mathbb{Z}^2_{0,N}: \exists \boldsymbol{i}\in\mathcal{Z}^{n-1}\text{ and }\boldsymbol{j}\in\mathcal{Z}^{0} \text{ such that } c_{\boldsymbol{i}, \boldsymbol{j}}\neq0\text{ and }\boldsymbol{k}=\boldsymbol{i}+\boldsymbol{j}\}.
\end{align*}
\end{definition}
Now, we state the main result of this paper as follows.
\begin{thm}
\label{23-1-13-0019}
For sufficiently large positive integer $N$,  if the driven model $\mathcal{K}_N$  is  hypoelliptic,  then  there exists $\epsilon_0>0$ such that for all $\epsilon\in(0,\epsilon_0)$, the stochastic flow   $\Phi$ of  \eqref{23-10-8-1953-1}-\eqref{23-10-8-1953-2} has full-horseshoes. Namely,  there exists     a pair of disjoint closed balls $\{U_1, U_2\}$ of $\mathbb{R}^d$   such that for $\mathbb{P}$-a.s. $\omega\in \Omega$,  there is a  subset $J(\omega)$ of $\mathbb{Z}_+:=\mathbb{N}\cup\{0\}$ with following properties,
\begin{enumerate}[(a)]
\item $\lim_{m\to+\infty}\frac{|J(w)\cap\{0,1.\dots,m-1\}|}{m}>0$;
\item for any $s\in \{1,2\}^{J(\omega)}$, there exists an $x_s\in\mathbb{R}^d$ such that $\Phi^j_{\omega}(x_s)\in U_{s(j)}$ for any $j\in J(\omega)$.
\end{enumerate}
\end{thm}

\begin{rem}
In  \cite{MR1846802}, it was proved that $\mathcal{K}_N=\{(0,1), (1,1)\}$ and $\mathcal{K}_N=\{(1,0), (1,1)\}$ are both  hypoelliptic for all $N\in\mathbb{N}$.
\end{rem}

\begin{rem}
In fact,  there exists a positive constant $b$ such that   the density of $J(\omega)$ is greater than $b$ for  $\mathbb{P}$-a.s. $\omega\in\Omega$.
\end{rem}


\begin{rem}
GSNS has     full-horseshoes of  any discrete time form. Namely,  for any $\tau\in(0,+\infty)$,  there exists   a pair of disjoint  closed balls $\{U_1, U_2\}$ of $\mathbb{R}^d$ such that for $\mathbb{P}$-a.s. $\omega\in\Omega$,  there is a  subset $J(\omega)$ of $\mathbb{Z}_+$ with following properties,
\begin{enumerate}[(a)]
\item $\lim_{m\to+\infty}\frac{|J(\omega)\cap\{0,1.\dots,m-1\}|}{m}>0$;
\item for any $s\in \{1,2\}^{J(\omega)}$, there exists an $x_s\in\mathbb{R}^d$ such that $\Phi^{j\tau}_{\omega}(x_s)\in U_{s(j)}$ for any $j\in J(\omega)$.
\end{enumerate}
\end{rem}

 Unfortunately, we can't describe accurately the location of full-horseshoes  at present.   We conjecture that GSNS has full-horseshoes on any two disjoint closed balls. Despite the  ideas of  our proof come from the known literature,  we  overcame the difficulties caused by  varied settings.

The organization of this paper is as follows:  In \Cref{22-12-05-01}, we review  basic knowledge of ergodic theory and entropy theory.  In \Cref{22-11-03-03}, we show that the stochastic flow of  GSNS  has positive entropy with respect to  its unique stationary measure, namely \Cref{23-2-2-2148}. In \Cref{22-03-08-01}, we prove the existence of measurable weak-horseshoes for GSNS, namely \Cref{weakhoreshoes}.  In \Cref{22-11-03-02}, we  extend the measurable weak-horseshoe to  the full-horseshoe. Then  \Cref{23-1-13-0019} is proved.

 \noindent{\bf Acknowledgments.}
 The second author would like to  thank  Alex Blumenthal   and Sam Punshon-Smith for useful discussions.  The authors were supported by NSFC of China (12090012,12090010,12031019,11731003).

 \section{Ergodic theory and entropy theory}
\label{22-12-05-01}


 In this section, we review some basic concepts and classical results  about measure-preserving dynamical systems,  measure of disintegration, relative entropy,  and relative Pinsker $\sigma$-algebra. The reader can see \cite{EW,MR0603625, G,Wal} for details.   

\subsection{Measure-preserving dynamical systems and measure of disintegration}
\label{22-07-06-01}
In this paper, we always work on  the \emph{Polish probability space} $(X,\mathscr{X},\mu)$, which means that  $X$ is a Polish space, $\mathscr{X}$ is the Borel-$\sigma$ algebra of $X$,  and $\mu$ is a probability measure on  $(X, \mathscr{X})$.  \emph{A measure-preserving dynamical system}   $(X,\mathscr{X},\mu,T)$ is said to be a measure-preserving map $T$ on the probability space $(X,\mathscr{X},\mu)$.  Given two measure-preserving dynamical systems  $(X,\mathscr{X},\mu,T)$ and $(Y,\mathscr{Y},\nu,S)$, we say that  $(Y,\mathscr{Y},\nu,S)$ is \emph{a factor} of $(X,\mathscr{X},\mu,T)$ if there exists a
measure-preserving map $\pi:(X,\mathscr{X},\mu)\rightarrow (Y,\mathscr{Y},\nu)$ such that $\pi\circ T=S\circ\pi$. And, $\pi$ is called a   \emph{factor map}.
\begin{definition}
\label{22-07-03-01}
Let $(X,\mathscr{X},\mu,T)$  be a measure-preserving dynamical system. It is called an \emph{ergodic}  measure-preserving dynamical system if $\mu(A)=1$ or $\mu(X\setminus A)=1$ wherever $A$ is a $T$-invariant measurable subset of $X$; it is called an \emph{invertible} measure-preserving dynamical system, if $T^{-1}: X\to X$ exists and is measurable.
\end{definition}

Let  $\pi:(X,\mathscr{X},\mu)\rightarrow (Y,\mathscr{Y},\nu)$ is a measure-preserving map between two  Polish probability spaces. Then,  there is a family of conditional probability measures $\{\mu_y\}_{y\in Y}$ on $(X,\mathscr{X})$ which are characterized by
\begin{itemize}
	\item  $\mu_y(\pi^{-1}(y))=1$ for $\nu$-a.s. $y\in Y$;
	\item for each $f \in L^1(X,\mathscr{X},\mu)$,  one  has that $f \in L^1(X,\mathscr{X},\mu_y)$ for $\nu$-a.s. $y\in Y$, the map $y \mapsto
	\int_X f\,\mathrm{d}\mu_y$ belongs to $L^1(Y,\mathscr{Y},\nu)$ and $\mu=\int_Y\mu_y\mathrm{d}\nu(y)$ in the sense that
	$$\int_Y \left(\int_X f\,\mathrm{d}\mu_y \right)\, \mathrm{d}\nu(y)=\int_X f \,\mathrm{d}\mu.$$
\end{itemize}
Then  $\mu=\int_Y\mu_y\mathrm{d}\nu(y)$  is called \emph{disintegration} of $\mu$ relative to $Y$.  Furthermore,  if $\pi:(X,\mathscr{X},\mu, T)\rightarrow (Y,\mathscr{Y},\nu, S)$ is a factor map between two invertible measure-preserving dynamical  systems on Polish probability spaces,  then $T_*\mu_y=\mu_{Sy}$  for $\nu$-a.s. $y\in Y$, where $T_*\mu_y$ is defined by 
	$$T_*\mu_y(A):=\mu_y(T^{-1}A),$$
	for any $A\in\mathscr{X}$.
	
\begin{lemma}[{\cite[Proposition 6.13]{EW}}]
\label{23-2-21-2054}
Let $\pi:(X,\mathscr{X},\mu,T)\rightarrow (Y,\mathscr{Y},\nu,S)$  a  factor map  between two measure-preserving systems on  Polish probability spaces.  Then 
$$(X\times X,\mathscr{X}\otimes \mathscr{X}, \mu \times_Y\mu, T\times T),$$ 
where the measure $\mu\times_Y\mu:=\int_{Y}(\mu_y\times \mu_y) \, d \nu(y)$, is   a  measure-preserving dynamical system.  
\end{lemma}



\subsection{Relative entropy and relative Pinsker $\sigma$-algebra}
In this subsection,  we always assume that $\pi: (X,\mathscr{X},\mu,T)\to (Z,\mathscr{Z},\eta, R)$ is a factor map between two invertible measure-preserving dynamical systems on Polish probability spaces.  And  we review  definitions of  its relative entropy and  relative Pinsker factor. 

For any two finite Borel measurable partitions $\alpha$ and $\beta$ of $X$,  denote $\alpha\vee\beta$ as the family of intersections of a set from $\alpha$ with a set from $\beta$, which is  a finite Borel measurable partition of $X$.  The definition of  multiple is similar. Next, we given a hierarchy of definitions of entropy:
\begin{align*}
& H_{\mu}(\alpha|\beta)=\sum_{A\in\alpha}\sum_{B\in\beta}-\mu(A\cap B)\log\frac{\mu(A\cap B)}{\mu(B)},
\\&h_{\mu}(T,\alpha)=\lim_{n\to+\infty}H_{\mu}(\alpha|\bigvee_{i=1}^{n}T^{-i}\alpha),
\\&h_{\mu}(T,\alpha| Z)=\int_Z h_{\mu_z}(T,\alpha)\mathrm{d}\eta(z),
\end{align*}
 where  $\mu=\int_{Z}\mu_z\mathrm{d}\eta(z)$  is the disintegration of $\mu$ relative to  $Z$.  Then  entropy  of $(X,\mathscr{X},\mu,T)$  relative to $Z$ is defined as 
 \begin{align}
 \label{23-2-12-22-16}
 h_{\mu}(T|Z)=\sup_{\alpha}h_{\mu}(T,\alpha|Z),
 \end{align}
where  $\alpha$ is taken all over finite Borel measurable partition of $X$.

  The \emph{relative Pinsker $\sigma$-algebra $\mathcal{P}_\mu(\pi)$} of the factor map $\pi: (X,\mathscr{X},\mu,T)\to (Z,\mathscr{Z},\eta, R)$ is defined as the smallest $\sigma$-algebra containing
 $$\{ A\in\mathscr{X}:  h_\mu(T,\{A, A^c\}| Z)=0\}.$$
  Note  $\mathcal{P}_\mu(\pi)$ is a $T$-invariant sub-$\sigma$-algebra of $\mathscr{X}$ (for example, see \cite[Section 4.10]{Wal} or \cite{MR1786718}). Hence,  it
 determines a measure-preserving dynamical system $(Y,\mathscr{Y},\nu,S)$ on the Polish probability space and  two factor maps
$$\pi_1: (X, \mathscr{X},\mu, T)\rightarrow (Y,\mathscr{Y},\nu,S),\quad\pi_2:(Y,\mathscr{Y},\nu,S)\rightarrow  (Z,\mathscr{Z},\eta, R),$$
 such that  $\pi_2\circ \pi_1=\pi$  and
$\pi_1^{-1}(\mathscr{Y})=\mathcal{P}_{\mu}(\pi)\pmod{\mu}$.  
The factor map  $\pi_1:(X, \mathscr{X},\mu, T)\rightarrow (Y,\mathscr{Y},\nu,S)$ is called \emph{relative Pinsker factor map} of $\pi$.

Let's end of this section by a  result about  conditional measure-theoretic entropy and relative Pinsker factor. The reader can refer to the proof for \cite[Lemma 4.1]{HL} and \cite[Lemma 3.3]{MR234267}.
\begin{lemma}
\label{22-10-19-06-01}
 Denote $\pi_1:(X,\mathscr{X},\mu,T)\rightarrow (Y,\mathscr{Y},\nu,S)$
as relative Pinsker factor map of $\pi:(X,\mathscr{X},\mu,T)\to (Z,\mathscr{Z},\eta, R)$. Then, for any $l\in\mathbb{N}$ and a finite Borel measurable partition $\alpha$ on $X$ one has that
\begin{enumerate}[(i)]
\item\label{22-12-12-2031-2}  $h_{\mu}(T^l,\alpha| Z)=h_{\mu}(T^l,\alpha|Y)=\int_Yh_{\mu_y}(T^l,\alpha)\mathrm{d}\nu(y)$;
\item\label{key-lem-4}  $\lim_{m\to+\infty}h_{\mu}(T^m,\alpha| Z)=H_{\mu}(\alpha|Y)$, where  $H_{\mu}(\alpha|Y):=\int_YH_{\mu_y}(\alpha)\mathrm{d}\nu(y)$.
\end{enumerate}
\end{lemma}

\section{Positive entropy of GSNS}
\label{22-11-03-03}
In this section,  we show that the stochastic flow of GSNS  has positive  entropy with respect to its unique stationary measure by borrowing the result in \cite{bedrossian2021chaos} and    Pesin's entropy formula on non-compact Riemannina manifolds in \cite{MR3175262}. 
\subsection{Invariant measure, entropy  and Lyapunov exponent for RDS}
In this subsection, we mainly  review some basic definitions in discrete random dynamical systems (being abbreviated as RDS). A canonical  model is generated  by the time-1 map of  solutions of classical SDEs (see \cite[Chapter \uppercase\expandafter{\romannumeral5}]{MR1369243} for details). 
 \begin{definition}
Let   $(\Omega, \mathscr{F},\mathbb{P},\theta)$  be a measure-preserving dynamical system on the Polish probability space.  A RDS $F$ on a Polish space $M$ over $(\Omega, \mathscr{F},\mathbb{P},\theta)$ means that 
 $$F: \mathbb{Z}_+\times\Omega\times M\to M,\quad (n,\omega, x)\mapsto F^n_{\omega}x$$
 is Borel measurable satisfying that  for $\mathbb{P}$-a.s. $\omega\in\Omega$, $F_{\omega}^0=\text{Id}_M$ and $F^{n+m}_{\omega}=F^n_{\theta^m\omega}\circ F^m_{\omega}$ wherever $n,m\in\mathbb{Z}_+$. Furthermore, if  for any $n\in\mathbb{Z}_+$ and $\mathbb{P}$-a.s. $\omega\in\Omega$, $F^n_{\omega}$ is continuous, then $F$ is called  a \emph{continuous RDS}.
\end{definition}
The RDS $F$  always is viewed as a skew product map given by
$$T: \Omega\times M\to \Omega\times M\quad\text{with}\quad  T(\omega, x)=(\theta\omega,  F^1_{\omega}x).$$
A  Borel probability measure $\mu$ on $\Omega\times M$ is called an \emph{invariant measure} of RDS $F$ if a)  $\mu$ is   $T$-invariant;  b)  $(\pi_{\Omega})_*\mu=\mathbb{P}$, where $\pi_{\Omega}$ is the projection from $\Omega\times M$ to $\Omega$.  Additionally, if $\mu$ is $T$-ergodic, then $\mu$ is called an \emph{invariant ergodic measure} of  $F$. 

\begin{definition}
For any an invaraint measure $\mu$ of RDS $F$, the entropy of $(F,\mu)$  is defined as 
$$h_{\mu}(F):=\sup_{\alpha}\lim_{n\to+\infty}\frac{1}{n}\int_{\Omega}H_{\mu_{\omega}}\left(\bigvee_{i=0}^{n-1}(F_{\omega}^{i})^{-1}\alpha\right)\mathrm{d}\mathbb{P}(\omega),$$
where $\alpha$ is taken over the set of all finite Borel measurable partitions of $M$,  and $\mu=\int_{\Omega}\delta_{\omega}\times\mu_{\omega}d\mathbb{P}(\omega)$ is the disintegration of $\mu$ relative to $\Omega$.  
\end{definition}

\begin{rem}
\label{23-2-24-2006}
The definition of  entropy of $(F,\mu)$ coincides with the definition of relative entropy, i.e. $h_{\mu}(F)=h_{\mu}(T|\Omega)$ (for example, see \cite[Theorem 2.3.4]{Bo11}).
\end{rem}
In smooth dynamical systems, Lyapunov exponent is a significant index to  measure  the divergence and convergence the speed of nearby trajectories by constructing  invariant manifolds.  The celebrated multiplicative ergodic theorem  asserts the existence of Lyapunov exponents.  We give it a convenient version in our setting.  The reader can see the details in \cite[Theorem 3.4.1]{A}.

\begin{proposition}
\label{22-12-2-01}
Suppose that  $F$ is  a RDS on $\mathbb{R}^d$ over an ergodic measure-preserving dynamical system $(\Omega,\mathscr{F},\mathbb{P},\theta)$ on the Polish probability space  and $\mu$ is an invariant ergodic measure of $F$.   If  $F_{\omega}^1$ is a $C^1$ diffeomorphsim on $\mathbb{R}^d$ for $\mathbb{P}$-a.s. $\omega\in\Omega$ and the 
  following integral condition hold,
  \begin{align}
  \int_{\Omega\times \mathbb{R}^d}\left(\log^+|\mathrm{d}_x F^1_{\omega}|+\log^+|\mathrm{d}_x(F_{\omega}^1)^{-1}|\right)\mathrm{d}\mu(\omega, x)<+\infty,
  \end{align}
  where $\log^+|a|:=\max\{\log|a|, 0\}$.  Then there is a $\mu$-full measure  subset $\Lambda\subset\Omega\times \mathbb{R}^d$ and  $d$ constants  $$+\infty>\lambda_1\geq\lambda_2\geq\cdots\geq\lambda_d>-\infty$$
such that for each  $(\uo,x)\in\Lambda$, there exists  a deceasing measurable filtration 
$$T_x\mathbb{R}^d=E_1(\uo,x)\supset E_2(\uo,x)\supset\cdots\supset E_d(\uo,x)\supsetneq E_{d+1}(\omega,x):=\{\boldsymbol{0}\},$$
with  following properties
\begin{enumerate}[(a)]
\item $(\mathrm{d}_xF^n_{\omega})E_i(\uo,x)=E_i(\theta^n\omega, F^n_{\omega}x)$  for  $i=1,2,\dots,d$, $n\in\mathbb{Z}_+$ and $(\omega, x)\in\Lambda$;
\item for any $i=1,\dots,d$ and $(\omega,x)\in\Lambda$,    $v\in E_i(\uo,x)\setminus E_{i+1}(\omega,x)$ if and only if 
$$\lim_{n\to+\infty}\frac{\log|(\mathrm{d}_xF^n_{\omega})v|}{n}=\lambda_i;$$
\end{enumerate}
Usually, $\lambda_1,\lambda_2,\dots,\lambda_d$ are called  Lyapunov exponents of $(F,\mu)$.  Particularly, $\lambda_1$ is called top Lyapunov exponent  of $(F,\mu)$.
\end{proposition}
\subsection{Stochastic flow}
In this subsection, we mainly bulit  the relationship between  discrete RDS and stochastic flow, and give some corresponding definitions for stochastic flow.
Let $\Phi$ be the stochastic flow on $\mathbb{R}^d$ over $(\Omega,\mathscr{F},\mathbb{P})$ defined as \eqref{23-2-18-1254}. It  naturally induces a family of Markov processes  on $\mathbb{R}^d$ whose transition probabilities $\mathcal{P}_t(x,\cdot)$, where $x\in \mathbb{R}^d$ and $t\geq0$, are defined by 
$$\mathcal{P}_t(x, A)=\mathbb{P}\big(\{\omega\in\Omega: \Phi^t_{\omega}(x)\in A\}\big)\text{ for any Borel subset $A$ of $\mathbb{R}^d$}.$$
 \begin{definition}
 A Borel  probability measure $\varrho$ on $\mathbb{R}^d$ is called a \emph{stationary measure} of $\Phi$,  if for any $t\geq 0$ and  any Borel subset $A$ of $\mathbb{R}^d$, one has that
$$\varrho(A)=\int_{\mathbb{R}^d}\mathcal{P}_t(x, A)\mathrm{d}\varrho(x).$$
Furthermore, $\varrho$  is  called an \emph{ergodic} stationary measure  if  for any  Borel subset $A$ of $\mathbb{R}^d$ satisfying that for all $t>0$, $\mathcal{P}_t(x, A)=1_A(x)$ for $\varrho$-a.s. $x\in\mathbb{R}^d$ is $\varrho$-null measure or full-measure (see the other equivalent definitions in \cite[Theorem 3.2.4]{MR1417491}).   
\end{definition}

For any $\tau\in(0,+\infty)$,  the time-$\tau$ map  of the stochastic flow $\Phi$ defines a  discrete RDS  $\Phi^{(\tau)}$ on $\mathbb{R}^d$ over $(\Omega,\mathscr{F},\mathbb{P},\theta^{\tau})$\footnote{Note that $(\Omega,\mathscr{F},\mathbb{P},\theta^{\tau})$ is an  ergodic measure-preserving dynamical system. The reader can refer to a similar proof in \cite[Appendix A.3]{A}.}, where $\theta^{\tau}$ is the Wiener shift. Particularly,  the RDS $\Phi^{(\tau)}$ is defined as following
\begin{align}
\label{23-10-9-1721}
\Phi^{(\tau)}:\mathbb{Z}_+\times \Omega\times\mathbb{R}^d\to\mathbb{R}^d,\quad (n,\omega,x)\mapsto \Phi^{n\tau}_{\omega}x. 
\end{align}

Next,  we give a lemma (see  \cite[Theorem 2.1 in Section I]{MR884892}) which illustrates  connections between stationary measure  of the stochastic flow generated  and invariant measure of RDS.
\begin{lemma}
\label{23-10-10-1528}
Let $\Phi$ be   the stochastic flow  on $\mathbb{R}^d$ over $(\Omega,\mathscr{F},\mathbb{P})$ defined as \eqref{23-2-18-1254}. If  $\varrho$ is a stationary measure, then 
for  any $\tau\in(0,+\infty)$,   $\mathbb{P}\times\varrho$ is an  invariant  measure of RDS  $\Phi^{(\tau)}$, i.e.  $\mathbb{P}\times\varrho$ is  invariant  with respect to the  skew product map  
\begin{align}
\label{23-2-21-2042}
T^{(\tau)}: \Omega\times \mathbb{R}^d\to \Omega\times \mathbb{R}^d\quad\text{with}\quad  T^{(\tau)}(\omega, x)=(\theta^{\tau}\omega,  \Phi^{\tau}_{\omega}x), 
\end{align}
 and $(\pi_{\Omega})_*\mu=\mathbb{P}$, where $\pi_{\Omega}$ is the projection from $\Omega\times\mathbb{R}^d$ to $\Omega$.
\end{lemma}

 Finally, we give  definition of  entropy  and Lyapunov exponents of stochastic flow with respect to its stationary measure.  The reader can refer to  \cite[Section 3 in Chapter \uppercase\expandafter{\romannumeral5}]{MR1369243} for more details. For sake of convenience,  we only consider the behavior  of  time-1 map of the stochastic flow in this paper.   
 \begin{definition}
Let  $\Phi$ be  the stochastic flow on $\mathbb{R}^d$ over $(\Omega,\mathscr{F},\mathbb{P})$ defined as \eqref{23-2-18-1254}
 and   $\varrho$ be  a stationary measure.  Then  $\mathbb{P}\times\varrho$  is an invariant measure of  RDS $\Phi^{(1)}$  by \Cref{23-10-10-1528}.   The entropy of $(\Phi,\varrho)$  is defined as 
 $$h_{\varrho}(\Phi):=h_{\mathbb{P}\times\varrho}(\Phi^{(1)}).$$
 If $\mathbb{P}\times\varrho$ is an invariant ergodic measure of RDS $\Phi^{(1)}$,  then the Lyapunov exponents of  $(\Phi,\varrho)$ are defined as the Lyapunov exponents of RDS $(\Phi^{(1)},\mathbb{P}\times\varrho)$ .
\end{definition}

\subsection{Positive entropy for GSNS} 
To prove that GSNS has positive entropy,  we need to recall   the Pesin's entropy formula.   It  builds   the connection measure-theoretic entropy  and positive Lyapunov exponents
  for the SRB measure on compact Riemannian manifold, such as  \cite{MR693976,MR968818, MR1369243,MR0466791} for a  variety of settings.   Note that the stationary measure $\varrho$ of  the stochastic flow of GSNS  has full-support (for example, see \cite{MR3862850, MR3382587}). Hence that, we need to consider   Pesin's  entropy formula for RDS on  $\mathbb{R}^d$. 

\begin{lemma}
\label{23-2-16-1653}
Let  $\Phi$ be  the stochastic flow on $\mathbb{R}^d$ over $(\Omega,\mathscr{F},\mathbb{P})$ defined as \eqref{23-2-18-1254} and $\varrho$ be  a smooth  stationary measure $\varrho$ of $\Phi$, i.e. $\varrho\ll  m_{\mathbb{R}^d}$ and $\mathbb{P}\times\varrho$ is an ergodic measure of RDS $\Phi^{(1)}$.  If    following three assumptions hold,
\begin{align}
\tag{\textbf{Assumption 1}}&\int_{\Omega\times\mathbb{R}^d}\log^+\sup_{v\in O(\boldsymbol{0}, 1)} |\mathrm{d}_{x+v}\Phi^n_{\omega}|\mathrm{d}(\mathbb{P}\times\varrho)<+\infty\quad \text{for any } n\in\mathbb{N},
\\\tag{\textbf{Assumption 2}}&\int_{\Omega\times\mathbb{R}^d}\log^+\big(\sup_{v\in O(\boldsymbol{0},1)}|\mathrm{d}^2_{x+v}\Phi^1_{\omega}|\big)\mathrm{d}(\mathbb{P}\times\varrho)<+\infty,
\\\tag{\textbf{Assumption 3}}&\int_{\Omega\times\mathbb{R}^d}\log^+ |\mathrm{d}_{\Phi^1_{\omega}(x)}\big(\Phi_{\omega}^1\big)^{-1}|+\log^+\big(\sup_{v\in O(\boldsymbol{0}, 1)}|\mathrm{d}^2_{\Phi^1_{\omega}(x+v)}\big(\Phi_{\omega}^1\big)^{-1}|\big)\mathrm{d}(\mathbb{P}\times\varrho)<+\infty,
\end{align}
where $O(\boldsymbol{0},1):=\{v\in\mathbb{R}^d: |v|<1\}$, then
$$h_{\varrho}(\Phi)=\sum_{i=1}^d\lambda_i^+,$$
where  $\lambda_1,\cdots,\lambda_d$ are the Lyapunov exponents of $(\Phi,\varrho)$ and $\lambda_i^+=\max\{\lambda_i,0\}$ for $i=1,\dots,d$.
\end{lemma}
\begin{proof}
In fact, above lemma is just  a   rewrite  of \cite[Theorem 3.8]{MR3175262} in our setting.  For completeness, we give a quick proof.  Let $(\tilde{\Omega},\tilde{\mathscr{F}},\tilde{\mathbb{P}})$ be a Polish probability space, where $\tilde{\Omega}=\text{Diff}^{\infty}(\mathbb{R}^d)$ endowed with relative  compact open topology and $\tilde{\mathbb{P}}$  is the distribution of  random variable
$\omega\mapsto \Phi^1_{\omega}.$  
Due to the properties of stochastic flow, we know that 
 \begin{align}
 \label{22-10-30-02}
 \Pi: (\Omega, \mathscr{F},\mathbb{P},\theta^1)\to (\tilde{\Omega}^{\mathbb{N}},\tilde{\mathscr{F}}^{\mathbb{N}},\tilde{\mathbb{P}}^{\mathbb{N}},\tilde{\theta}),\quad \omega\mapsto (\Phi^1_{\omega}, \Phi^1_{\theta^1\omega},\dots)
 \end{align}
 is a  factor map between two measure-preserving dynamical systems on the Polish probability spaces, where the $\tilde{\theta}$ is left-shit map on $\tilde{\Omega}^{\mathbb{N}}$. 
Now,  we  define  a  i.i.d. RDS  $\tilde{\Phi}^{(1)}$ on  $\mathbb{R}^d$ over $(\tilde{\Omega}^{\mathbb{N}},\mathscr{F}^{\mathbb{N}},\tilde{\mathbb{P}}^{\mathbb{N}},\tilde{\theta})$ as following 
\begin{align*}
\tilde{\Phi}^{(1)}: \mathbb{Z}_+\times\tilde{\Omega}\times\mathbb{R}^d\to\mathbb{R}^d\quad (n,\tilde{\omega}_{\infty},x)\mapsto  
\begin{cases}
\tilde{\Phi}_{\tilde{\omega}_{\infty}}^nx:=\tilde{\omega}_{\infty}(n)\circ\cdots\circ\tilde{\omega}_{\infty}(1)x ,\quad&\text{if } n\in\mathbb{N},
\\\tilde{\Phi}_{\tilde{\omega}_{\infty}}^0x:=x,\quad&\text{if } n=0,
  \end{cases}
\end{align*}
By assumptions of stochastic flow $\Phi$ and \eqref{22-10-30-02}, it is not hard to see   that 
\begin{itemize}
\item $\varrho$ is also the ergodic stationary measure of i.i.d. RDS $\tilde{\Phi}^{(1)}$;
\item  the Lyapunov exponents and entropy of i.i.d. RDS  $(\tilde{\Phi}^{(1)},\tilde{\mathbb{P}}^{\mathbb{N}}\times\varrho)$ are equal to  the Lyapunov exponents and entropy of RDS  $(\Phi^{(1)},\mathbb{P}\times\varrho)$, respectively;
\item following three integral conditions hold:
\begin{align}
\notag&\int_{\tilde{\Omega}^{\mathbb{N}}\times\mathbb{R}^d}\log^+\sup_{v\in O(\boldsymbol{0}, 1)} |\mathrm{d}_{x+v}\tilde{\Phi}_{\tilde{\omega}_{\infty}}^n|\mathrm{d}(\tilde{\mathbb{P}}^{\mathbb{N}}\times\varrho)<+\infty\quad \text{for any } n\in\mathbb{N},
\\\notag&\int_{\tilde{\Omega}^{\mathbb{N}}\times\mathbb{R}^d}\log^+\big(\sup_{v\in O(\boldsymbol{0},1)}|\mathrm{d}^2_{x+v}\tilde{\Phi}_{\tilde{\omega}_{\infty}}^1|\big)\mathrm{d}(\tilde{\mathbb{P}}^{\mathbb{N}}\times\varrho)<+\infty,
\\\notag&\int_{\tilde{\Omega}^{\mathbb{N}}\times\mathbb{R}^d}\log^+ |\mathrm{d}_{\tilde{\Phi}^1_{\tilde{\omega}_{\infty}}(x)}\big(\tilde{\Phi}_{\tilde{\omega}_{\infty}}^1\big)^{-1}|+\log^+\big(\sup_{v\in O(\boldsymbol{0}, 1)}|\mathrm{d}^2_{\tilde{\Phi}^1_{\tilde{\omega}_{\infty}}(x+v)}\big(\tilde{\Phi}_{\tilde{\omega}_{\infty}}^1\big)^{-1}|\big)\mathrm{d}(\tilde{\mathbb{P}}^{\mathbb{N}}\times\varrho)<+\infty.
\end{align}
\end{itemize} Therefore,  by \cite[Theorem 3.8]{MR3175262} one has that $h_{\varrho}(\Phi)=h_{\mathbb{P}\times\varrho}(\Phi^{(1)})=\sum_{i=1}^d\lambda_i^+.$
\end{proof}

\begin{proposition}
\label{23-2-2-2148}
For  any  positive integer  $N\geq 392$, if $\mathcal{K}_N$ is  hypoelliptic, then  there exists $\epsilon_0>0$ such that for all $\epsilon\in(0,\epsilon_0)$  the stochastic flow  $\Phi$ of \eqref{23-10-8-1953-1}-\eqref{23-10-8-1953-2}  has a   unique stationary measure $\varrho$ ,  and 
$$h_{\varrho}(\Phi)=\sum_{i=1}^d\lambda_i^+\geq\lambda_1>0,$$
where $\lambda_1,\cdots,\lambda_d$ are the Lyapunov exponents of $(\Phi,\varrho)$.
 \end{proposition}
\begin{proof}
Firstly, the existence of unique stationary measure for stochastic flow $\Phi$ was proved in \cite{MR1417491}. Denoting this unique stationary measure as $\varrho$, then 
\begin{lemma}
\label{23-2-12-1854}
The stationary measure  $\varrho$  of $\Phi$  has  following properties 
\begin{enumerate}[(a)]
\item\label{23-10-9-1634-1}  the unique stationary measure $\varrho$ is also an ergodic stationary measure of the stochastic flow $\Phi$;
\item\label{23-10-9-1634-2} denoting  $m_{\mathbb{R}^d}$ as the volume measure on $\mathbb{R}^d$,  $\varrho\ll m_{\mathbb{R}^d}$ and  $\rho:=\frac{d\varrho}{dm_{\mathbb{R}^d}}$ is a smooth positive function satisfying that there exist some constants $C, \eta>0$  such that for every $x\in\mathbb{R}^d$,
$\rho(x)\leq Ce^{-\eta|x|^2}$;
\item\label{23-10-9-1634-3}  for any $\tau\in(0,+\infty)$  $\mathbb{P}\times\varrho$ is an  invariant ergodic measure of RDS  $\Phi^{(\tau)}$, i.e.  $\mathbb{P}\times\varrho$ is  invariant and   ergodic with respect to the  skew product map  
\begin{align}
\label{23-2-21-2042}
T^{(\tau)}: \Omega\times \mathbb{R}^d\to \Omega\times \mathbb{R}^d\quad\text{with}\quad  T^{(\tau)}(\omega, x)=(\theta^{\tau}\omega,  \Phi^{\tau}_{\omega}x), 
\end{align}
 and $(\pi_{\Omega})_*\mu=\mathbb{P}$, where $\pi_{\Omega}$ is the projection from $\Omega\times\mathbb{R}^d$ to $\Omega$.
\end{enumerate}
\end{lemma}
\begin{proof}
\eqref{23-10-9-1634-1} follows from  from \cite[Theorem 3.2.6]{MR1417491}.  According to \cite[Theorem 1.5 and Corollary 1.6 ]{MR4408018} and \cite[Theorem 2.1 in Section I]{MR884892}, \eqref{23-10-9-1634-2} and \eqref{23-10-9-1634-3} hold.
\end{proof}

Next,  we aim to verify the \textbf{Assumption 1-Assumption 3} in \Cref{23-2-16-1653} for stochastic flow $\Phi$.  To simply  estimations,  write \eqref{23-10-8-1953-1}-\eqref{23-10-8-1953-2} as 
\begin{align}
\label{22-11-21-01}
 \dot{x}=B(x, x)-\epsilon Ax+\sum_{i=1}^{d}e_i\dot{W}^i,
 \end{align}
 where $B:\mathbb{R}^d\times\mathbb{R}^d\to\mathbb{R}^d$ is bilinear,  $A$ is a  $d\times d$ matrix,   $e_i$ is the constant vector and $W^i$ is the one-dimensional Wiener process for $i=1,\cdots, d$.

Verification for \textbf{Assumption 1:} For any $i\in\{1,2\dots, d\}$, $t\in (0,+\infty)$ and $x\in\mathbb{R}^d$, the equality
\begin{align*}
\partial_{i}\Phi^t_{\omega}(x)&=\partial_{i}x+\int_0^tB\big(\Phi^{\tau}_{\omega}(x), \partial_{i}\Phi^{\tau}_{\omega}(x)\big)+B\big(\partial_{i}\Phi^{\tau}_{\omega}(x),\Phi^{\tau}_{\omega}(x)\big)\mathrm{d}\tau-\int_0^t\epsilon A\big(\partial_{i}\Phi^{\tau}_{\omega}(x)\big)\mathrm{d}\tau,
\end{align*}
holds for any  $\mathbb{P}$-a.s. $\omega\in \Omega$, which implies that
$$|\partial_{i}\Phi^t_{\omega}(x)|\leq1+\int_0^t\Big(a+b|\Phi^{\tau}_{\omega}(x)|\Big)|\partial_{i}\Phi^{\tau}_{\omega}(x)|\mathrm{d}\tau,$$
where $a>1$  and $b>1$ are  constants which only depend on $A$ and  $B$, respectively. For example, we can take $a=\max\{\|A\|,2\}$  and $b=\max\{\|B\|,2\}$. By using  Gronwall's inequality (for example, see \cite[Lemma 1.1]{oguntuase2001inequality}),   we have
\begin{align}
\label{22-11-02-02}
 |\partial_{i}\Phi^t_{\omega}(x)|\leq \exp\left(\int_0^ta+b|\Phi^{\tau}_{\omega}(x)|\mathrm{d}\tau\right).
 \end{align}
 Let  $\{v_k\}_{k\in\mathbb{N}}$ be a countable dense subset of  $O(\boldsymbol{0},1):=\{v\in\mathbb{R}^d: |v|<1\}$.   Fixing a $t\in(0,+\infty)$,  let $f_m(\omega,x )=\sup_{i=1,\cdots m}|\Phi_{\omega}^t(x+v_i)|$ be a Borel measurable function for any $m\in\mathbb{N}$.  Define  a Borel measurable partition of $\Omega\times\mathbb{R}^d$ as 
\begin{align*}
&O_1:=\big\{(\omega,x)\in \Omega\times\mathbb{R}^d:  |\Phi^t_{\omega}(x+v_i)|=f_m(\omega,x )\big\},
\\&O_k:=\big\{(\omega,x)\in \Omega\times\mathbb{R}^d:  |\Phi^t_{\omega}(x+v_k)|=f_m(\omega,x )\big\}\setminus\bigcup_{l=1}^{k-1}O_l \text{ for } k=2,\dots,m.
\end{align*}
Then, one has that 
\begin{align*}
\int_{\Omega\times\mathbb{R}^d}f_m(\omega, x)\mathrm{d}(\mathbb{P}\times\varrho)&=\sum_{k=1}^m\int_{O_k}|\Phi^{t}_{\omega}(x+v_k)|\mathrm{d}(\mathbb{P}\times\varrho)
\\&=\sum_{k=1}^m\int_{T^{(t)}O_k}|x+v_k|\mathrm{d}(\mathbb{P}\times\varrho)
\\&\leq\sum_{k=1}^m\int_{T^{(t)}O_k}(|x|+1)\mathrm{d}(\mathbb{P}\times\varrho)
\\&\overset{\text{\eqref{23-10-9-1634-2} in \Cref{23-2-12-1854}}}\leq\int_{\Omega\times\mathbb{R}^d}|x|\mathrm{d}(\mathbb{P}\times\varrho)+1<+\infty,
\end{align*}
where $T^{(t)}$  is an Borel  measurable transformation on $\Omega\times\mathbb{R}^d$ defined as \eqref{23-2-21-2042}. It follows that 
\begin{align*}
\int_{\Omega\times\mathbb{R}^d}\sup_{v\in O(\mathbf{0},1)}|\Phi^t_{\omega}(x+v)|\mathrm{d}(\mathbb{P}\times\varrho)&=\int_{\Omega\times\mathbb{R}^d}\lim_{m\to+\infty}f_m(\omega,x)\mathrm{d}(\mathbb{P}\times\varrho)
\\&=\lim_{m\to+\infty}\int_{\Omega\times\mathbb{R}^d}f_m(\omega,x)\mathrm{d}(\mathbb{P}\times\varrho)
\\&\leq\int_{\Omega\times\mathbb{R}^d}|x|\mathrm{d}(\mathbb{P}\times\varrho)+1.
\end{align*}
Therefore,  for any $n\in\mathbb{N}$
\begin{align}
\notag\int_{\Omega\times\mathbb{R}^d}\log^+\sup_{v\in O(\boldsymbol{0},1)} |\partial_{i}\Phi^n_{\omega}(x+v)|\mathrm{d}(\mathbb{P}\times\varrho)&\overset{\eqref{22-11-02-02}}{\leq} a+b\int_{\Omega\times\mathbb{R}^d}\int_0^n\sup_{v\in O(\boldsymbol{0},1)}|\Phi^{\tau}_{\omega}(x+v)|\mathrm{d}\tau\mathrm{d}(\mathbb{P}\times\varrho)
\\\notag&=a+b\int_0^n\int_{\Omega\times\mathbb{R}^d}\sup_{v\in O(\boldsymbol{0},1)}|\Phi^{\tau}_{\omega}(x+v)|\mathrm{d}(\mathbb{P}\times\varrho)\mathrm{d}\tau
\\\notag&\leq a+b\int_0^n\int_{\Omega\times\mathbb{R}^d}|x|+1\mathrm{d}(\mathbb{P}\times\varrho)\mathrm{d}\tau
\\\label{23-2-16-2214}&\overset{\text{\eqref{23-10-9-1634-2} in \Cref{23-2-12-1854}}}= a+nb+nb\int_{\mathbb{R}^d}|x|\mathrm{d}\varrho(x)<+\infty.
\end{align}
 It follows that \textbf{Assumption 1} holds.

Verification for \textbf{Assumption 2:} For any $i,j\in\{1,2\dots, d\}$, $t\in (0,+\infty)$ and  $x\in \mathbb{R}^d$, following  equality hods  for $\mathbb{P}$-a.s. $\omega\in\Omega$, 
\begin{align*}
\partial_{ij}\Phi^t_{\omega}(x)&=\int_0^t\left(B\big(\Phi^{\tau}_{\omega}(x), \partial_{ij}\Phi^{\tau}_{\omega}(x)\big)+B\big(\partial_{ij}\Phi^{\tau}_{\omega}(x),\Phi^{\tau}_{\omega}(x)\big)\right)\mathrm{d}\tau
\\&\quad+\int_0^t\left(B\big(\partial_{i}\Phi^{\tau}_{\omega}(x),\partial_{j}\Phi^{\tau}_{\omega}(x)\big)+B\big(\partial_{j}\Phi^{\tau}_{\omega}(x),\partial_{i}\Phi^{\tau}_{\omega}\big)-\epsilon A\big(\partial_{ij}\Phi^{\tau}_{\omega}(x)\big)\right)\mathrm{d}\tau
\end{align*}
which implies that
\begin{align*}
|\partial_{ij}\Phi^1_{\omega}(x)|&\leq\int_0^1(a+b|\Phi^{\tau}_{\omega}(x)|)\cdot|\partial_{ij}\Phi^{\tau}_{\omega}(x)|\mathrm{d}\tau+b\int_0^1|\partial_{j}\Phi^{\tau}_{\omega}(x)|\cdot|\partial_{i}\Phi^{\tau}_{\omega}(x)| \mathrm{d}\tau
\\&\overset{\eqref{22-11-02-02}}{\leq}\int_0^1(a+b|\Phi^{\tau}_{\omega}(x)|)\cdot|\partial_{ij}\Phi^{\tau}_{\omega}(x)|\mathrm{d}\tau+b\exp\left(2\int_0^1a+b|\Phi^{\tau}_{\omega}(x)|\mathrm{d}\tau\right).
\end{align*}
Using Gronwall's inequality again,   we have
\begin{align}
\label{22-1-9-1417}
 |\partial_{ij}\Phi^{1}_{\omega}(x)|\leq b\exp\left(3\int_0^1a+b|\Phi^{\tau}_{\omega}(x)|\mathrm{d}\tau\right).
\end{align}
 Combing  \eqref{23-2-16-2214} and \eqref{22-1-9-1417}, one has that 
\begin{align*}
\int_{\Omega\times\mathbb{R}^d}\sup_{v\in O(\boldsymbol{0},1)}\log^+ |\partial_{ij}\Phi^1_{\omega}(x)|\mathrm{d}(\mathbb{P}\times\varrho)\leq& (b+3a)+ 3b\int_{\Omega\times\mathbb{R}^d}\int_0^1\sup_{v\in O(\boldsymbol{0},1)}|\Phi^{\tau}_{\omega}(x)|\mathrm{d}\tau\,\mathrm{d}(\mathbb{P}\times\varrho)
\\\leq&(b+3a)+ 3b\int_0^1\int_{\Omega\times\mathbb{R}^d}\sup_{v\in O(\boldsymbol{0},1)}|\Phi^{\tau}_{\omega}(x)|\mathrm{d}\tau\,\mathrm{d}(\mathbb{P}\times\varrho)
\\<&+\infty.
\end{align*}
Therefore,  $\int_{\Omega\times\mathbb{R}^d}\log^+\left(\sup_{v\in O(\boldsymbol{0},1)}|\mathrm{d}^2_{x+v}\Phi^1_{\omega}|\right)\mathrm{d}(\mathbb{P}\times\varrho)<+\infty$.

Verification for \textbf{Assumption 3:} Using the backward flow of \Cref{22-11-21-01},  for any $t\geq 0$ and  $x\in\mathbb{R}^d$,  one has that 
\begin{align*}
(\Phi^{t}_{\omega})^{-1}x&=x-\int_{0}^tB\big(\Phi^{\tau}_{\omega}\circ(\Phi^t_{\omega})^{-1}x, \Phi^{\tau}_{\omega}\circ(\Phi^t_{\omega})^{-1}x\big)-\epsilon A\big(\Phi^{\tau}_{\omega}\circ(\Phi^t_{\omega})^{-1}x\big)\mathrm{d}\tau-\sum_{i=1}^{d}\int_0^te_i\mathrm{d}W^i_{\tau}.
\end{align*}
for  $\mathbb{P}$-a.s. $\omega\in\Omega$.   Repeating above argument in verification for \textbf{Assumption 1} and \textbf{Assumption 2},   one  can verify that
$$\int_{\Omega\times\mathbb{R}^d}\log^+|\mathrm{d}_{\Phi^1_{\omega}(x)}\big(\Phi^1_{\omega}\big)^{-1}|+\log^+\big(\sup_{v\in O(\boldsymbol{0}, 1)}|\mathrm{d}^2_{\Phi^1_{\omega}(x+v)}\big(\Phi^1_{\omega}\big)^{-1}|\big)\mathrm{d}(\mathbb{P}\times\varrho)<+\infty.$$

By \Cref{23-2-16-1653}, all that remains  to be  proved  is that $\Phi$ has positive top Lyapunov exponent  with respect to the stationary measure $\varrho$.  This result is proved by Bedrossian and Punshon-Smith. The reader can see the details in \cite[Theorem 1.1 and Remark 1.5]{bedrossian2021chaos}.  For all,  this finishes the proof of  \Cref{23-2-2-2148}.
\end{proof}

\section{Revisit of weak-horseshoes}
\label{22-03-08-01}

In this section, we review the proof of  \cite[Theorem 2.1]{HL} to obtain  measurable  weak-horseshoes for GSNS, namely, \Cref{weakhoreshoes}.  The measurable property originates from  the property  disintegration of measure and return times.

\subsection{Preparing lemmas}

In this subsection, we mainly present some preparing  lemmas.   Following lemma is summarized from \cite[Lemma 3.4, Lemma 3.5, Lemma 4.3]{HL}.
\begin{lemma} \label{key-lem}
Assume that $\pi:(X,\mathscr{X},\mu,T)\rightarrow (Z,\mathscr{Z},\eta,R)$ is  a  factor map between two  invertible  ergodic  measure-preserving dynamical systems on Polish probability spaces. Let  $\pi_1:(X,\mathscr{X},\mu,T)\rightarrow
(Y,\mathscr{Y},\nu ,S)$ be the relative Pinsker factor map of 
$\pi$ and  $\mu=\int_Y  \mu_y d \nu(y)$ be the disintegration of $\mu$ relative to $Y$. If  $h_\mu(T|Z)>0$, then
\begin{enumerate}
\item\label{key-lem-1} $\mu_y$ is non-atomic (that is $\mu_y(\{x\})=0$ for each $x\in  X$) for $\nu$-a.s. $y\in Y$;
\item\label{key-lem-2} $(X\times X, \mathscr{X}\otimes\mathscr{X}, \mu\times_Y\mu, T\times T)$ is an ergodic measure-preserving dynamical system on  the Polish probability space, where
$$(X\times X,\mathscr{X}\otimes \mathscr{X},\mu \times_Y\mu, T\times T)$$
 is the product of $(X,\mathscr{X},\mu,T)$ with itself relative to $\pi$ (see \Cref{23-2-21-2054});
\item\label{key-lem-3} if $U_1, U_2\in\mathscr{X}$ with $\mu\times_Y\mu(U_1\times U_2)>0$,  then there exists a Borel measurable subset $\boldsymbol{A}$ of  $Y$ with $\nu(\boldsymbol{A})>0$, a positive integer $\boldsymbol{r}>2$ and a Borel measurable partition
$\alpha=\{B_1,\dots, B_{\boldsymbol{r}}\}$ of $X$ such that $\pi_1^{-1}(\boldsymbol{A})\cap B_i\subset U_i, i=1,2$ and $\mu_y(B_j)=1/\boldsymbol{r}, j=1,\dots, \boldsymbol{r}$ for $\nu$-a.s. $y\in Y$.
\end{enumerate}
 \end{lemma}

Following combinatorial  lemma is well-known which is a  Karpovsky-Milman-Alon’s generalization of the Sauer-Perles-Shelah lemma. For example, the reader can  see it in  \cite[Corollary 1]{MR710891}.
\begin{lemma}
\label{22-10-10-01}
Given a positive integer $\boldsymbol{r}\geq 2$ and $\boldsymbol{\lambda}\in(1,+\infty)$, there exists a positive constant $\boldsymbol{c}$ such that for all $n\in\mathbb{N}$, if $\mathcal{R}\subset\{1,\cdots,r\}^{\{1,\dots,n\}}$
satisfies $|\mathcal{R}|\geq \big((\boldsymbol{r}-1)\boldsymbol{\lambda}\big)^n$, then there is a subset  $J(n,\mathcal{R})$ of $\{1,\dots,n\}$ with $|J(n,\mathcal{R})|\geq \boldsymbol{c}n$ such that for any $s\in \{1,\dots, \boldsymbol{r}\}^{J(n,\mathcal{R})}$, there exists $\check{s}\in\mathcal{R}$ with $s(j)=\check{s}(j)$ for each $j\in J(n,\mathcal{R})$.
\end{lemma}

Now, we introduce a  variation measurable selection theorem. It is a quick result of  \cite[Lemma 5.25]{EW} by using measure-preserving property. 
\begin{lemma}
\label{22-10-19-06}
Let $\pi: (X,\mathscr{X},\mu)\to(Y,\mathscr{Y},\nu)$ be a factor map between two Polish  probability spaces. Then there exists a Borel measurable map $\hat{\pi}: Y\to X$ such that
$\pi\circ\hat{\pi}(y)=y$ for $\nu$-a.s. $y\in Y$.
\end{lemma}

\subsection{Weak-horseshoes}
Before stating  \Cref{weakhoreshoes}, we give  a one to one corresponding relationship between the element in $\{0,1\}^{\mathbb{Z}_+}$ and  subset of $\mathbb{Z}_+$  given by
 \begin{align}
 \label{23-1-5-1819}
 u\in\{0,1\}^{\mathbb{Z}_+}\mapsto \hat{u}=\{n\in\mathbb{Z}_+: u(n)=1\}.
 \end{align}
   Now, let's display  measurable weak-horseshoes of GSNS.
\begin{thm}
\label{weakhoreshoes}
Under the setting in  \Cref{23-2-2-2148}, there is  a pair of   non-empty
disjoint compact subsets $\{U_1, U_2\}$ of $\mathbb{R}^d$, a positive constant $\boldsymbol{b}$, a sequence  $\{\mathbf{N}_n\}_{n\in\mathbb{N}}$ of strictly  increasing Borel measurable maps $\mathbf{N}_n: \Omega\to\mathbb{Z}_+$\footnote{In this place, "strictly increasing" means that $\mathbf{N}_{n+1}(\omega)>\mathbf{N}_n(\omega)$ for any $n\in\mathbb{N}$ and $\mathbb{P}$-a.s. $\omega\in\Omega$.},  and a sequence $\{\gamma_n\}_{n\in\mathbb{N}}$ of Borel measurable  maps $\gamma_n: \Omega\to \{0,1\}^{\mathbb{Z}_+}$ such that for any $n\in\mathbb{N}$ and $\mathbb{P}$-a.s. $\omega\in\Omega$, one has that
\begin{enumerate}[(a)]
\item $\hat{\gamma}_n(\omega)\subset\{0,1,\dots, \mathbf{N}_n(\omega)-1\}$ and  $|\hat{\gamma}_n(\omega)|\geq \boldsymbol{b}\mathbf{N}_n(\omega)$;
\item  for any $s\in\{1,2\}^{\hat{\gamma}_n(\omega)}$, there exists an $x_s\in \mathbb{R}^d$ with $ \Phi^j_{\omega}(x_s)\in U_{s(j)}$ for any $j\in \hat{\gamma}_n(\omega)$.
\end{enumerate}
\end{thm}

\begin{proof}
Recall that $\Phi$ is the stochastic flow of GSNS on $\mathbb{R}^d$ over $(\Omega,\mathscr{F},\mathbb{P})$ defined as \eqref{23-2-18-1254}, $\varrho$ is the unique stationary measure of $\Phi$,  and $\theta^1: \Omega\to\Omega$  is the Winer shift with $\theta^1(\omega)=\omega(\cdot+1)-\omega(\cdot)$.
It is well-known that $(\Omega,\mathscr{F},\mathbb{P},\theta^1)$ is an ergodic measure-preserving dynamical system on the Polish probability space.

Let  $(\bar{\Omega},\bar{\mathscr{F}},\bar{\mathbb{P}},\bar{\theta} )$  be the invertible extension of $(\Omega,\mathscr{F},\mathbb{P},\theta)$, i.e.  $(\bar{\Omega},\bar{\mathscr{F}},\bar{\mathbb{P}},\bar{\theta} )$ is an invertible ergodic measure-preserving dynamical system such that  there is a factor map $\bar{\pi}_{\Omega}: (\bar{\Omega},\bar{\mathscr{F}},\bar{\mathbb{P}},\bar{\theta})\to (\Omega,\mathscr{F},\mathbb{P},\theta^1)$ (see \cite{MR0143873,MR0217258}).   Denoting   $X:=\bar{\Omega}\times \mathbb{R}^d$ and $\mathscr{X}$ as its Borel $\sigma$-algebra,  define a Borel measurable map $T$ on $\bar{\Omega}\times \mathbb{R}^d$ as following 
\begin{align}
T: X\to X, \quad (\bar{\omega}, x)\mapsto (\bar{\theta}\bar{\omega},  \Phi^1_{\bar{\pi}_{\Omega}\bar{\omega}}x).
\end{align}
Since \cite[Theorem 1.7.2]{A} and $\mathbb{P}\times\varrho$ is an invariant ergodic measure with respect to $T^{(1)}$ (see \eqref{23-10-9-1634-3} in \Cref{23-2-12-1854}), there exists a unique Borel measure $\mu$ on $(X,\mathscr{X})$  satisfying that 
 $$(\bar{\pi}_{\Omega}\times\text{Id}_{\mathbb{R}^d})_*\mu=\mathbb{P}\times\varrho,$$  
 and $(X,\mathscr{X},\mu, T)$ is an ergodic measure-preserving dynamical system.    Then
 $$\pi_{\bar{\Omega}}:\big(X, \mathscr{X},\mu, T\big)\rightarrow (\bar{\Omega},  \bar{\mathscr{F}},\bar{\mathbb{P}},\bar{\theta})\text{ with } (\bar{\omega},x)\mapsto \bar{\omega}$$ is a factor map between  two invertible ergodic measure-preserving dynamical systems and $h_{\mu}(T|\pi_{\bar{\Omega}})=h_{\mathbb{P}\times\varrho}(\Phi^{(1)})=h_{\varrho}(\Phi)>0$,  since \Cref{23-2-2-2148} and \Cref{23-2-24-2006}.

 Denote $\mathcal{P}_{\mu}(\pi_{\bar{\Omega}})$ as the relative Pinsker $\sigma$-algebra of $\pi_{\bar{\Omega}}$.  
 Then, there exists an invertible  measure-preserving dynamical system  $(Y,\mathscr{Y},\nu,S)$  on the  Polish probability space and  two factor maps
$$\pi_1: (X, \mathscr{X},\mu, T)\rightarrow (Y,\mathscr{Y},\nu,S),\quad\pi_2:(Y,\mathscr{Y},\nu,S)\rightarrow  (\bar{\Omega}, \bar{\mathscr{F}}, \bar{\mathbb{P}}, \bar{\theta}),$$ between
invertible measure-preserving dynamical systems on Polish probability spaces such that  $\pi_2\circ \pi_1=\pi_{\bar{\Omega}}$  and
$\pi_1^{-1}(\mathscr{Y})=\mathcal{P}_{\mu}(\pi_{\bar{\Omega}})\pmod{\mu}$.   

Let  $\mu=\int_Y\mu_y d \nu(y)$ be the disintegration relative to $Y$.    By \eqref{key-lem-1} and \eqref{key-lem-2} in \Cref{key-lem},  we know that the
measure-preserving dynamical system
$$(X\times X, \mathscr{X}\times\mathscr{X}, \mu\times_Y\mu, T\times T),$$
is ergodic and $\mu_y$ is non-atomic for $\nu$-a.s. $y\in Y$.  Next, we give a sufficient condition to ensure the existence of measurable full-horseshoes.
\begin{lemma}
\label{22-12-13-35}
For any a pair of disjoint closed balls $\{U_1, U_2\}$ of $\mathbb{R}^d$, if 
$$\mu\times_Y\mu\big((\bar{\Omega}\times U_1)\times (\bar{\Omega}\times U_2)\big)>0, $$
 then there exists a positive constant $\boldsymbol{b}$, a sequence $\{\mathbf{N}_n\}_{n\in\mathbb{N}}$ of strictly  increasing Borel measurable maps $\mathbf{N}_n:\bar{\Omega}\to\mathbb{N}$ and a sequence $\{\gamma_n\}_{n\in\mathbb{N}}$ of Borel measurable  maps $\gamma_n: \bar{\Omega}\to \{0,1\}^{\mathbb{Z}_+}$
such that for any $n\in\mathbb{N}$ and $\bar{\mathbb{P}}$-a.s. $\bar{\omega}\in\bar{\Omega}$ one has that
\begin{enumerate}[(a)]
\item $\hat{\gamma}_n(\bar{\omega})\subset\{0,1,\dots, \mathbf{N}_n(\bar{\omega})-1\}$ and  $|\hat{\gamma}_n(\bar{\omega})|\geq \boldsymbol{b}\mathbf{N}_n(\omega)$;
\item  for any $s\in\{1,2\}^{\hat{\gamma}_n(\bar{\omega})}$, there exists an $x_s\in \mathbb{R}^d$ with $ \Phi^j_{\bar{\omega}}(x_s)\in U_{s(j)}$ for any $j\in \hat{\gamma}_n(\bar{\omega})$.
\end{enumerate}
\end{lemma}
Now we are ready to prove  \Cref{weakhoreshoes} assuming  \Cref{22-12-13-35}. Due to \Cref{22-10-19-06}, it is clear that we only need to explain the  existence of $\{U_1, U_2\}$ in \Cref{22-12-13-35}. The proof of \Cref{22-12-13-35} will be 
carried later.

Let $\pi_{\mathbb{R}^d}: \bar{\Omega}\times \mathbb{R}^d\to \mathbb{R}^d$ be the projection and
\begin{align*}
&\Delta_{\pi_{\mathbb{R}^d}}:=\left\{ \big((\bar{\omega}_1,x_1),(\bar{\omega}_2,x_2)\big)\in X\times X: x_1=x_2\right\},
\\&\Delta_{X}:=\{\big((\bar{\omega}_1,x_1),(\bar{\omega}_2,x_2)\big)\in X\times X: (\bar{\omega}_1,x_1)=(\bar{\omega}_2,x_2)\}.
\end{align*}
Note that $\pi_1^{-1}(y)\subset\pi_{\bar{\Omega}}^{-1}(\pi_2(y))$ and $\mu_y(\pi_1^{-1}(y))=1$. It follows that 
\begin{align*}
\mu_y\times\mu_y(\Delta_{\pi_{\mathbb{R}^d}})&=\mu_y\times\mu_y\left(\Delta_{\pi_{\mathbb{R}^d}}\cap\big(\pi_{\bar{\Omega}}^{-1}(\pi_2(y))\times\pi_{\bar{\Omega}}^{-1}(\pi_2(y))\big)\right)
\\&\leq \mu_y\times\mu_y(\Delta_X)
\\&\overset{ \text{\Cref{key-lem} \eqref{key-lem-1}}}=0
\end{align*}
for $\nu$-a.s. $y\in Y$.    Hence, $\mu\times_Y\mu(X\times X\setminus \Delta_{\pi_{\mathbb{R}^d}})=1$.  Denote
\begin{align*}
&\mathfrak{O}:=\{\bar{O}(x, 1/n):  x\in\mathbb{Q}^d\text{ and }n\in\mathbb{N}\}, \text{ and }\mathfrak{U}:=\{(U_1, U_2)\in \mathfrak{O}\times\mathfrak{O}: \rho(U_1, U_2)>0\},
\end{align*}
where $\bar{O}(x, 1/n):=\{y\in\mathbb{R}^d: |x-y|\leq 1/n\}$, and  $\rho(U_1, U_2)=\inf_{x\in U_1,y\in U_2}|x-y|$ with $|\cdot|$ being standard norm on $\mathbb{R}^d$. 
It is clear that 
 \begin{align*}
 X\times X\setminus\Delta_{\pi_{\mathbb{R}^d}} =&\bigcup_{(U_1, U_2)\in\mathfrak{U}} (\bar{\Omega}\times U_1)\times (\bar{\Omega}\times U_2).
 \end{align*}
Therefore, there exists $\{U_1, U_2\}\in\mathfrak{U}$  such that $\mu\times_Y\mu\big((\bar{\Omega}\times U_1\big)\times\big(\bar{\Omega}\times U_2)\big)>0.$
This finishes the proof of \Cref{weakhoreshoes}.
\end{proof}

Finally, we divide into four steps to prove \Cref{22-12-13-35}.
\begin{proof}[Proof of \Cref{22-12-13-35}]

\textbf{Step 1:}  In this step, we mainly introduce function of estimation entropy   and recurrence time sequence.  The notations will be fixed through this proof and we use  boldface form to stress  some important notations.

Recall that  $\{U_1, U_2\}$ is a pair of disjoint closed balls of $\mathbb{R}^d$ with
$$\mu\times_Y\mu\Big((\bar{\Omega}\times U_1)\times (\bar{\Omega}\times U_2)\Big)>0.$$
By \eqref{key-lem-3} in Lemma \ref{key-lem}, there  exists a Borel measurable subset $\boldsymbol{A}$ of $Y$ with $\nu(\boldsymbol{A})>0$,
a positive integer $\boldsymbol{r}>2$, and a Borel measurable partition $\alpha=\{B_1, B_2,\dots, B_{\boldsymbol{r}}\}$ of $X$ such that
\begin{align}
\label{22-12-13-1138}
\pi_1^{-1}(\boldsymbol{A})\cap B_i\subset \Omega\times U_i  \text{ for }  i=1,2,
\end{align} and $\mu_y(B_j)=1/\boldsymbol{r}$, $j=1,\dots, \boldsymbol{r}$ for $\nu$-a.s. $y\in Y$.

\textbf{Function of estimation entropy:}  According to \eqref{key-lem-4} in \Cref{22-10-19-06-01}, we have
\begin{align*}
\lim_{l\to+\infty}h_{\mu}(T^l, \alpha|\bar{\Omega})&=H_{\mu}(\alpha|Y)=\sum_{j=1}^{\boldsymbol{r}}\int_Y-\mu_y(B_j)\log\mu_y(B_j)\mathrm{d}\nu(y)=\log\boldsymbol{r}.
\end{align*}
Thus, there is an $\boldsymbol{l}\in\mathbb{N}$  such that 
$$3\boldsymbol{c}_0:=h_{\mu}(T^{\boldsymbol{l}}, \alpha|\bar{\Omega})-\nu(Y\setminus \boldsymbol{A})\log \boldsymbol{r}- \nu(\boldsymbol{A})\log(\boldsymbol{r}-1)>0.$$

\begin{claim}
\label{23-2-11-1615}
 Define  a  Borel measurable function on $Y$ as follows,
\begin{align*}
\mathbf{\delta}(y):=\Big(h_{\mu_y}(T^{\boldsymbol{l}},\alpha)-\log(\boldsymbol{r}-1)\Big)1_{\boldsymbol{A}}(y),
 \end{align*}
 where $h_{\mu_y}(T^{\boldsymbol{l}},\alpha):=\lim_{n\to+\infty}H_{\mu_y}\Big(\alpha| \bigvee_{i=1}^nT^{-i\boldsymbol{l}}\alpha\Big)$.  Then there exists a  Borel measurable subset $\boldsymbol{D}$ of $Y$ such that $\nu(\boldsymbol{D})>0$ and,  for each $y\in \boldsymbol{D}$ one has that:
\begin{enumerate}[(i)]
\item $\lim_{m\to+\infty}\frac{1}{m}\sum_{i=0}^{m-1}1_{\boldsymbol{A}}(S^{\boldsymbol{l}i}y)$ exists, and it is greater than $0$.  Denoting the limiting  as $1_{\boldsymbol{A}}^*(y)$, then $1_{\boldsymbol{A}}^*(y)>0$;
\item $\lim_{m\to+\infty}\frac{1}{m}\sum_{i=0}^{m-1}\delta(S^{\boldsymbol{l}i}y)$ exists, and it is  greater than or equal to $3\boldsymbol{c}_0$. Denoting the limiting  as $\delta^*(y)$, then $\delta^*(y)\geq 3\boldsymbol{c}_0$;
\item for any $i\in\mathbb{Z},$ $(T^i)_*\mu_y=\mu_{S^iy}$.
\item\label{23-3-3-2157} for any $i\in\mathbb{Z}$, one has that 
$$1_{\boldsymbol{A}}(S^{\boldsymbol{l}i}y)\geq \frac{\mathbf{\delta}(S^{\boldsymbol{l}i})}{\log\boldsymbol{r}-\log(\boldsymbol{r}-1)}.$$
\end{enumerate}
\end{claim}
\begin{proof}[Proof of \Cref{23-2-11-1615}]
Note that   $\nu(\boldsymbol{A})>0$,   $h_{\mu_y}( T^{\boldsymbol{l}},\alpha)\leq\log\boldsymbol{r}$  for $\nu$-a.s. $y\in Y$,  and 

\begin{align*}
\int_Y\delta(y)d\nu(y)&=\int_{\boldsymbol{A}}h_{\mu_y}(T^{\boldsymbol{l}},\alpha)\mathrm{d}\nu(y)-\nu(\boldsymbol{A})\log(\boldsymbol{r}-1)
\\&=\int_Yh_{\mu_y}(T^{\boldsymbol{l}},\alpha)\mathrm{d}\nu(y)-\int_{Y\setminus \boldsymbol{A}}h_{\mu_y}(T^{\boldsymbol{l}},\alpha)\mathrm{d}\nu(y)-\nu(\boldsymbol{A})\log(\boldsymbol{r}-1)
\\&\overset{\text{\eqref{22-12-12-2031-2} in \Cref{22-10-19-06-01}}}{\geq} h_{\mu}(T^{\boldsymbol{l}},\alpha|\Omega)-\nu(Y\setminus \boldsymbol{A})\log \boldsymbol{r}-\nu(\boldsymbol{A})\log(\boldsymbol{r}-1)=3\boldsymbol{c}_0.
\end{align*}
Then  the existence of $\boldsymbol{D}$ directly follows from Birkhoff ergodic theorem and  the property of measure of  disintegration.
    \end{proof}

 \textbf{Recurrence time sequence:}    Given $n\in\mathbb{N}$, define a map  $\boldsymbol{a}_n: \boldsymbol{D}\to \mathbb{Z}_+:=\mathbb{N}\cup\{0\} $ being the $n_{th}$ time return to $\boldsymbol{A}$ function under the transformation $S$ with $\boldsymbol{l}|\boldsymbol{a}_n(y)$ for any $y\in \boldsymbol{D}$, i.e. 
 $$ \boldsymbol{a}_n(y):=\inf\{\boldsymbol{l}(m-1): \sum_{k=0}^{m-1} 1_A(S^{\boldsymbol{l}k}y)=n\text{ and }m\in\mathbb{N}\}.$$ 
  Since for each $y\in \boldsymbol{D}$
 $$\lim_{m\to+\infty}\frac{1}{m}\sum_{k=0}^{m-1}1_{\boldsymbol{A}}(S^{\boldsymbol{l}k}y)=1_{\boldsymbol{A}}^*(y)>0,$$ 
$\boldsymbol{a}_n$ is well-defined for each $n\in\mathbb{N}$. For each $y\in\boldsymbol{D}$,  denoting 
 $$\mathfrak{a}(y)=\{\boldsymbol{a}_1(y), \boldsymbol{a}_2(y), \cdots\},$$
 then  we have
\begin{align}
\notag\lim_{m\to+\infty}\frac{|\mathfrak{a}(y)\cap\{0,1,\dots, m-1\}|}{m}&=\lim_{m\to+\infty}\frac{|\mathfrak{a}(y)\cap\{0,1,\dots, \boldsymbol{l}(m-1)\}|}{\boldsymbol{l}m}
\\\notag&=\lim_{m\to+\infty}\frac{1}{\boldsymbol{l}m}\sum_{k=0}^{m-1}1_{\boldsymbol{A}}(S^{\boldsymbol{l}k}y)
\\\notag&\overset{\text{ \Cref{23-2-11-1615} \eqref{23-3-3-2157}}}\geq\lim_{m\to+\infty}\frac{1}{\boldsymbol{l}m}\sum_{k=0}^{m-1}\frac{\delta(S^{\boldsymbol{l}k}y)}{\log \boldsymbol{r}-\log(\boldsymbol{r}-1)}
\\\notag&=\frac{\delta^*(y)}{\boldsymbol{l}\log\frac{\boldsymbol{r}}{\boldsymbol{r}-1}}
\\\label{22-10-10-02}&\geq  \frac{3\boldsymbol{c}_0}{\boldsymbol{l}\log\frac{\boldsymbol{r}}{\boldsymbol{r}-1}}=:3\boldsymbol{c}_1.
\end{align}

Next, we will show that $\{\boldsymbol{a}_n\}_{n\in\mathbb{N}}$  is  a family of  Borel measurable maps $\boldsymbol{a}_n: \boldsymbol{D}\to\mathbb{Z}_+$.   Noting that  $\{y\in \boldsymbol{D}: \boldsymbol{a}_1(y)=0\}=\boldsymbol{D}\cap \boldsymbol{A}$  and for any $k\in\mathbb{N}$
$$\{y\in \boldsymbol{D}:  \boldsymbol{a}_1(y)=\boldsymbol{l}k\}= \boldsymbol{D}\cap S^{-\boldsymbol{l}k}\boldsymbol{A}\cap\left(\bigcap_{i=0}^{k-1} S^{-\boldsymbol{l}i}(Y\setminus\boldsymbol{A})\right),$$
then $\boldsymbol{a}_1$ is Borel measurable.  For any $n\in\mathbb{N}$, suppose that $\boldsymbol{a}_n$ is Borel measurable.  Then $\boldsymbol{a}_{n+1}$ is also Borel measurable, because
$$\{y\in \boldsymbol{D}: \boldsymbol{a}_{n+1}(y)=\boldsymbol{l}n\}=\boldsymbol{D}\cap \left(\bigcap_{i=0}^{n} S^{-\boldsymbol{l}i}\boldsymbol{A}\right)\in\mathscr{Y},$$ and for any $k\geq n+1$
\begin{align*}
&\{y\in \boldsymbol{D}: \boldsymbol{a}_{n+1}(y)=\boldsymbol{l}k\}
\\&=\boldsymbol{D}\cap S^{-\boldsymbol{l}k}\boldsymbol{A}\cap \left(\bigcup_{m=n}^{k-1}\bigcap_{i=m}^{k-1} S^{-\boldsymbol{l}i}(Y\setminus \boldsymbol{A})\cap\{y\in\boldsymbol{D}: \boldsymbol{a}_n(y)=\boldsymbol{l}(m-1)\}\right)\in \mathscr{Y}.
\end{align*}
Thus,  $\{\boldsymbol{a}_n\}_{n\in\mathbb{N}}$ is a family of  Borel measurable maps.

Note that $\boldsymbol{E}:=\bigcup_{i=0}^{+\infty}S^{-i}\boldsymbol{D}$
 is a $\nu$-full measure subset of $Y$, since $(Y,\mathscr{Y},\nu, S)$ is an ergodic measure-preserving  dynamical system and $\nu(\boldsymbol{D})>0$.  Let 
 \begin{itemize}
\item $\tau: \boldsymbol{E}\to\mathbb{Z}_+:=\mathbb{N}\cup\{0\}$ be the first time return to $\boldsymbol{D}$  under the transformation $S$, i.e. 
 $$\tau(z):=\inf\{n\in\mathbb{Z}_+:  S^nz\in\boldsymbol{D}\};$$ 
  \item  $\{\boldsymbol{b}_n\}_{n\in\mathbb{N}}$ be a sequence of Borel measurable maps 
 $$\boldsymbol{b}_n: \boldsymbol{E}\to \mathbb{Z}_+ \text{ with } \boldsymbol{b}_n(z)=\tau(z)+\boldsymbol{a}_n( S^{\tau(z)}z).$$ 
 \end{itemize}
 It is clear that 
 \begin{enumerate}[(I)]
 \item $\{\boldsymbol{b}_n\}_{n\in\mathbb{N}}$ is a sequence of Borel measurable maps $\boldsymbol{b}_n: \boldsymbol{E}\to\mathbb{Z}_+$;
 \item  denoting  $\mathfrak{b}(z)=\{\boldsymbol{b}_1(z), \boldsymbol{b}_2(z),\dots\},$  then by  \eqref{22-10-10-02} one has that for each $z\in \boldsymbol{E}$, 
\begin{align}
\label{22-10-13-01}
\lim_{m\to+\infty}\frac{|\mathfrak{b}(z)\cap\{0,1,\dots, m-1\}|}{m}\geq3\boldsymbol{c}_1;
\end{align}
\item  for any  $z\in \boldsymbol{E}$, one has that  $S^{j}z \in \boldsymbol{A}$ for any $j\in\mathfrak{b}(z)$.
 \end{enumerate}

\textbf{Step 2}:  Recall that $\{B_1,\cdots, B_{\boldsymbol{r}}\}$ is a finite Borel measurable partition of  of $X$, which is define in \textbf{Step 1}. In this step,  we mainly use   positive entropy property of systems  to estimate the "hitting freedom" of the elements $\{B_1,\cdots, B_{\boldsymbol{r}}\}$ along $\mathfrak{b}(z)$ for each  $z\in\boldsymbol{E}$. Precisely, it is formulated as  following claim.

\begin{claim}
\label{22-10-12-03}
There exists  a  Borel measurable map $\tilde{\mathbf{N}}_0: \boldsymbol{E}\to\mathbb{N}$   such that for any $z\in \boldsymbol{E}$ and $k\geq \tilde{\mathbf{N}}_0(z)$,  one has that $|\mathcal{R}_k(z)|\geq\big((\boldsymbol{r}-1)2^{\boldsymbol{c}_0}\big)^k$, where
\begin{align}
\mathcal{R}_k(z):=\left\{s\in\{1,\cdots, \boldsymbol{r}\}^{\{1,\dots, k\}}: \mu_z\left(\pi_1^{-1}z\cap\Big(\bigcap_{j=1}^kT^{-\boldsymbol{b}_j(z)}B_{s(j)}\Big)\right)>0\right\}.
\end{align}
\end{claim}
\begin{proof}[Proof of \Cref{22-10-12-03}]
For any  $k\in\mathbb{N}$, define a map  $\delta_k: \boldsymbol{D}\to\mathbb{R}$ with
$$\delta_k(y):=\frac{1}{\frac{\boldsymbol{a}_k(y)}{\boldsymbol{l}}+1}\sum_{j=0}^{\boldsymbol{a}_k(y)/\boldsymbol{l}}\delta(S^{\boldsymbol{l}j}y).$$   It is clear that $\delta_k$ is  Borel measurable,  since $\boldsymbol{a}_k: \boldsymbol{D}\to\mathbb{Z}_+$ is Borel mesurable. Due to $\lim_{k\to+\infty} \delta_k(y)=\delta^*(y)\geq 3\boldsymbol{c}_0$  for any $y\in \boldsymbol{D}$,   one has that
\begin{align}
\label{22-11-12-01}
\boldsymbol{D}=\bigcup_{i=1}^{\infty}\bigcap_{k=i}^{\infty}\boldsymbol{D}_k,
\end{align}
where $\boldsymbol{D}_k:=\{y\in \boldsymbol{D}: \delta_k(y)\geq\boldsymbol{c}_0\}$.   Define a map $\mathbf{\check{N}}_0: \boldsymbol{D}\to\mathbb{N}$ such that $\mathbf{\check{N}}_0(y)$ is  the smallest positive integer $n$ satisfying $y\in\bigcap_{k=n}^{\infty}\boldsymbol{D}_k$. For any $n\in\mathbb{N}$,    it is clear that
$$\{y\in \boldsymbol{D}:  \mathbf{\check{N}}_0(y)=n\}=\left(\bigcap_{k=n}^{\infty}\boldsymbol{D}_k\right)\bigcap\left(\boldsymbol{D}\setminus\bigcap_{k=n-1}^{\infty}\boldsymbol{D}_k\right),$$
where $\boldsymbol{D}_0:=\boldsymbol{D}$.  It follows that  $\mathbf{\check{N}}_0: \boldsymbol{D}\to \mathbb{N}$ is Borel measurable.

For each $k\in\mathbb{N}$  and   $z\in \boldsymbol{E}$, one has that 
\begin{align*}
\log|\mathcal{R}_k(z)|&\geq H_{\mu_z}\left(\bigvee_{j=1}^kT^{-\boldsymbol{b}_j(z)}\alpha\right)
\\&=H_{\mu_{ S^{-\tau(z)}\circ S^{\tau(z)} z}}\left(\bigvee_{j=1}^kT^{-\boldsymbol{a}_j(S^{\tau(z)}z)-\tau(z)}\alpha\right)
\\&=H_{\mu_{S^{\tau(z)}z}}\left(\bigvee_{j=1}^kT^{-\boldsymbol{a}_j(S^{\tau(z)}(z))}\alpha\right)
\\&=H_{\mu_{S^{\tau(z)}z}}\left(T^{-\boldsymbol{a}_1(S^{\tau(z)}z)}\alpha|\bigvee_{j=2}^kT^{-\boldsymbol{a}_j(S^{\tau(z)}z)}\alpha\right)+H_{\mu_{S^{\tau(z)}z}}\left(\bigvee_{j=2}^kT^{-\boldsymbol{a}_j(S^{\tau(z)}z)}\alpha\right)
\\&=H_{\mu_{S^{\boldsymbol{b}_1(z)}(z)}}\left(\alpha|\bigvee_{j=2}^kT^{-(\boldsymbol{a}_j(S^{\tau(z)}z)-\boldsymbol{a}_1(S^{\tau(z)}z))}\alpha\right)+H_{\mu_{S^{\tau(z)}(z)}}\left(\bigvee_{j=2}^kT^{-\boldsymbol{a}_j(S^{\tau(z)}z)}\alpha\right)
\\&\geq h_{\mu_{S^{\boldsymbol{b}_1(z)}z}}(T^{\boldsymbol{l}}, \alpha)+H_{\mu_{S^{\tau(z)}(z)}}\left(\bigvee_{j=2}^kT^{-\boldsymbol{a}_j(S^{\tau(z)}z)}\alpha\right)
\\&\cdots
\\&\geq\sum_{j=1}^kh_{\mu_{S^{\boldsymbol{b}_j(z)}z}}(T^{\boldsymbol{l}},\alpha).
\end{align*}
Therefore, for any $z\in \boldsymbol{E}$  and  $k\geq\mathbf{\check{N}}_0(S^{\tau(z)}z)$, one has  that
\begin{align*}
\log|\mathcal{R}_k(z)|&\geq\sum_{j=1}^kh_{\mu_{S^{\boldsymbol{b}_j(z)}z}}(T^{\boldsymbol{l}},\alpha)
\\&=k\log(\boldsymbol{r}-1)+\sum_{j=0}^{\boldsymbol{a}_k(S^{\tau(z)}z)/\boldsymbol{l}}\big(h_{\mu_{S^{\boldsymbol{l}j+\tau(z)}z}}(T^{\boldsymbol{l}},\alpha)-\log(\boldsymbol{r}-1)\big)1_{\boldsymbol{A}}(S^{\boldsymbol{l}j+\tau(z)}z)
\\&=k\log(\boldsymbol{r}-1)+\sum_{j=0}^{\boldsymbol{a}_k(S^{\tau(z)}z)/\boldsymbol{l}}\delta(S^{\boldsymbol{l}j+\tau(z)}z)
\\&\geq k\log(\boldsymbol{r}-1)+\Big(\frac{\boldsymbol{a}_k(S^{\tau(z)}z)}{\boldsymbol{l}}+1\Big)\boldsymbol{c}_0
\\&\geq k\log(\boldsymbol{r}-1)+k\boldsymbol{c}_0.
\end{align*}
For all, the \Cref{22-10-12-03} holds by letting $\tilde{\mathbf{N}}_0: \boldsymbol{E}\to\mathbb{N}$  with $z\mapsto\mathbf{\check{N}}_0(S^{\tau(z)}z)$.
\end{proof}

\textbf{Step 3}: In this step, we mainly use the combinatorial  lemma (\Cref{22-10-10-01}) to obtain a sequence of  measurable hitting times such that the elements in  $\{B_1,\cdots, B_{\boldsymbol{r}}\}$ can hit freely along these hitting times. This is,

\begin{claim}
\label{22-10-19-05}
There exists a positive constant $\boldsymbol{b}$, a sequence $\{\tilde{\mathbf{N}}_n\}_{n\in\mathbb{N}}$ of strictly monotone increasing Borel measurable maps $\tilde{\mathbf{N}}_n: \boldsymbol{E}\to\mathbb{Z}_+$,  and  a sequence  $\{\tilde{\gamma}_n\}_{n\in\mathbb{N}}$ of Borel  measurable maps $\tilde{\gamma}_n: \boldsymbol{E}\to\{0,1\}^{\mathbb{Z}_+}$ such that for each $z\in E$ and $n\in\mathbb{N}$,
\begin{enumerate}[(a)]
\item $\hat{\tilde{\gamma}}_n(z)\subset\{0,1,\dots, \tilde{\mathbf{N}}_n(z)-1\}$ and $|\hat{\tilde{\gamma}}_n(z)|\geq \boldsymbol{b}\tilde{\mathbf{N}}_n(z)$;
\item for any $s\in\{1,\dots, \boldsymbol{r}\}^{\hat{\tilde{\gamma}}_n(z)}$, one has that
$$\mu_z\left(\pi_1^{-1}z\cap\bigcap_{j\in \hat{\tilde{\gamma}}_n(z)}T^{-j}B_{s(j)}\right)>0.$$
\end{enumerate}
\end{claim}
\begin{proof}[Proof of \Cref{22-10-19-05}]

 Recall that $\boldsymbol{r}$ is a positive integer  and $\boldsymbol{c}_0$ is positive constant defined in \textbf{Step 1}, and  for any $z\in E$, $\mathfrak{b}(z)$ is a subset of $\mathbb{Z}_+$ defined in \textbf{Step 2}.   Let  $\boldsymbol{c}$ be a positive constant  defined by applying   \Cref{22-10-10-01}  to $\boldsymbol{r}$ and $\boldsymbol{\lambda}:=2^{\boldsymbol{c}_0}$.   Combining \eqref{22-10-13-01} and  the similar argument in the proof  of \Cref{22-10-12-03}, there exists a   Borel measurable map $\tilde{\mathbf{N}}_1: \boldsymbol{E}\to\mathbb{Z}_+$  such that for any $z\in \boldsymbol{E}$,  one has that $\tilde{\mathbf{N}}_1(z)\geq 1/\boldsymbol{c}$,  and for each $m\geq \tilde{\mathbf{N}}_1(z)$, 
\begin{align}
\label{22-12-13-21}
|\mathfrak{b}(z)\cap\{0,1,\dots, m-1\}|\geq \boldsymbol{c}_1m>\tilde{\mathbf{N}}_0(z),
\end{align}
where $\tilde{\mathbf{N}}_0: \boldsymbol{E}\to\mathbb{Z}_+$ is the  Borel measurable map  defined in \Cref{22-10-12-03} and $\boldsymbol{c}_1$ is the positive constant defined in \eqref{22-10-10-02}.

 For each $n\in\mathbb{N}$, define a map $\tilde{\mathbf{N}}_n: \boldsymbol{E}\to\mathbb{N}$ with $\tilde{\mathbf{N}}_n(z):=\tilde{\mathbf{N}}_1(z)+n-1$.  Denoting 
 \begin{align}
\label{22-10-19-07}
\check{\mathbf{N}}_n(z):=|\mathfrak{b}(z)\cap\{0,1,\dots, \tilde{\mathbf{N}}_n(z)-1\}|,
\end{align}
then     $\check{\mathbf{N}}_n$ is  Borel measurable, since $\{\boldsymbol{b}_n\}_{n\in\mathbb{N}}$ is a sequence of  Borel measurable maps $\boldsymbol{b}_n: \boldsymbol{E}\to\mathbb{Z}_+$.  Given $n\in\mathbb{N}$ and $\mathcal{R}\subset\{1,\cdots \boldsymbol{r}\}^{\{1,2,\dots,n\}}$ with $|\mathcal{R}|\geq\big((\boldsymbol{r}-1)\boldsymbol{\lambda}\big)^n$,  applying  \Cref{22-10-10-01} to $\boldsymbol{r}$, $\boldsymbol{\lambda}$ and $\boldsymbol{c}$, we can choose a set 
\begin{align}
J(n,\mathcal{R})\subset\{1,\dots,n\}
\end{align}
with $|J(n,\mathcal{R})|\geq\boldsymbol{c}n$ such that for any $s\in\{1,\dots,\boldsymbol{r}\}^{J(n,\mathcal{R})}$, there exists $\check{s}\in\mathcal{R}$ with $s(j)=\check{s}(j)$ for each $j\in J(n,\mathcal{R})$.  By \Cref{22-10-12-03},  one has that for $z\in\boldsymbol{E}$
$$\mathcal{R}_{\check{\mathbf{N}}_n(z)}(z)\subset\{1,\cdots,\boldsymbol{r}\}^{\{1,\cdots, \check{\mathbf{N}}_n(z)\}}\quad\text{and}\quad |\mathcal{R}_{\check{\mathbf{N}}_n(z)}(z)|\geq\big((\boldsymbol{r}-1)\boldsymbol{\lambda}\big)^{\check{\mathbf{N}}_n(z)}.$$
Therefore, there exists a subset\footnote{For any $z_1\neq z_2\in\boldsymbol{E}$, if $\check{\mathbf{N}}_n(z_1)=\check{\mathbf{N}}_n(z_2)$ and  $\mathcal{R}_{\check{\mathbf{N}}_n(z_1)}(z_1)=\mathcal{R}_{\check{\mathbf{N}}_n(z_2)}(z_2)$, we  desire that $J\big(\check{\mathbf{N}}_n(z_1), \mathcal{R}_{\check{\mathbf{N}}_n(z_1)}(z_1)\big)=J\big(\check{\mathbf{N}}_n(z_2), \mathcal{R}_{\check{\mathbf{N}}_n(z_2)}(z_2)\big)$.}
\begin{align}
\label{22-12-12-2243}
J\big(\check{\mathbf{N}}_n(z), \mathcal{R}_{\check{\mathbf{N}}_n(z)}(z)\big)\subset\{1,\cdots, \check{\mathbf{N}}_n(z)\}
\end{align}
with  $|J\big(\check{\mathbf{N}}_n(z), \mathcal{R}_{\check{\mathbf{N}}_n(z)}(z)\big)|\geq \boldsymbol{c}\check{\mathbf{N}}_n(z)\geq 1$  such that  for any $s\in \{1,\dots,\boldsymbol{r}\}^{J\big(\check{\mathbf{N}}_n(z), \mathcal{R}_{\check{\mathbf{N}}_n(z)}(z)\big)}$, 
$$\mu_z\left(\pi_1^{-1}z\cap\bigcap_{j\in J\big(\check{\mathbf{N}}_n(z), \mathcal{R}_{\check{\mathbf{N}}_n(z)}(z)\big)}T^{-\boldsymbol{b}_j(z)}B_{s(j)}\right)>0.$$
According to the corresponding relationship in  \eqref{23-1-5-1819},   we can define a map  
 $$\underline{\gamma}_n: \boldsymbol{E}\to\{0,1\}^{\mathbb{Z}_+}\quad\text{with}\quad \hat{\underline{\gamma}}_n(z):=J\big(\check{\mathbf{N}}_n(z),\mathcal{R}_{\check{\mathbf{N}}_n(z)}(z)\big).$$ 

Fix $n\in\mathbb{N}$. Now, we are going to prove that $\underline{\gamma}_n: \boldsymbol{E}\to\{0,1\}^{\mathbb{Z}_+}$ is Borel measurable.  Note  that the image of  $\underline{\gamma}_n$ contains at most countable points in $\{0,1\}^{\mathbb{Z}_+}$ by \eqref{22-12-12-2243}.
Hence that, we only need to prove that for any finite subset $J$ of $\mathbb{Z}_+$,
\begin{align}
\{z\in\boldsymbol{E}: \hat{\underline{\gamma}}_n(z)=J\}
\end{align}
is  a measurable subset of $\boldsymbol{E}$. It is sufficient to prove that  any  $\check{n}\in\mathbb{N}$ and $\mathcal{R}\subset \{1,\dots,\boldsymbol{r}\}^{\{1,\dots,\check{n}\}}$,
\begin{align*}
&\{z\in E: \check{\mathbf{N}}_n(z)=\check{n}\text{ and }\mathcal{R}_{\check{\mathbf{N}}_n(z)}(z)=\mathcal{R}\}
\\=&\{z\in E: \check{\mathbf{N}}_n(z)=\check{n}\}\bigcap \{z\in E: \mathcal{R}_{\check{n}}(z)=\mathcal{R}\}
\\=&\{z\in E: \check{\mathbf{N}}_n(z)=\check{n}\}\bigcap\left(\bigcap_{s\in \mathcal{R} }\left\{z\in\boldsymbol{E}: \mu_z\Big(\bigcap_{j=1}^{\check{n}}T^{-\boldsymbol{b}_j(z)}B_{s(j)}\Big)>0\right\}\right)
\\&\quad\bigcap\left(\bigcap_{s\in\{1,\dots,\boldsymbol{r}\}^{\{1,\dots,\check{n}\}}\setminus\mathcal{R} }\left\{z\in\boldsymbol{E}: \mu_z\Big(\bigcap_{j=1}^{\check{n}}T^{-\boldsymbol{b}_j(z)}B_{s(j)}\Big)=0\right\}\right)
\end{align*}
 is a Borel measurable subset of $Y$. Since for any $s\in\mathcal{R}$,
\begin{align*}
&\left\{z\in\boldsymbol{E}: \mu_z\Big(\bigcap_{j=1}^{\check{n}}T^{-\boldsymbol{b}_j(z)}B_{s(j)}\Big)>0\right\}
\\=&\bigcup_{0\leq b_1<\cdots<b_{\check{n}}}\left(\{z\in\boldsymbol{E}: \boldsymbol{b}_j(z)=b_j \text{ for } j=1,\dots,\check{n}\}\cap \left\{z\in\boldsymbol{E}: \mu_z\Big(\bigcap_{j=1}^{\check{n}}T^{-b_j}B_{s(j)}\Big)>0\right\}\right)
\end{align*}
and for any $s\in\{1,\dots,\boldsymbol{r}\}\setminus\mathcal{R}$,
\begin{align*}
&\left\{z\in\boldsymbol{E}: \mu_z\Big(\bigcap_{j=1}^{\check{n}}T^{-\boldsymbol{b}_j(z)}B_{s(j)}\Big)=0\right\}
\\=&\bigcup_{0\leq b_1<\cdots<b_{\check{n}}}\left(\{z\in\boldsymbol{E}: \boldsymbol{b}_j(z)=b_j \text{ for } j=1,\dots,\check{n}\}\cap \left\{z\in\boldsymbol{E}: \mu_z\Big(\bigcap_{j=1}^{\check{n}}T^{-b_j}B_{s(j)}\Big)=0\right\}\right)
\end{align*}
 both  are Borel measurable subsets of $Y$, $\underline{\gamma}_n$ is Borel measurable.

For each $n\in\mathbb{N}$,  define $\tilde{\gamma}_n: \boldsymbol{E}\to\{0,1\}^{\mathbb{Z}_+}$ with
$$\hat{\tilde{\gamma}}_n(z)=\{\boldsymbol{b}_j(z): j\in \hat{\underline{\gamma}}_n(z)\}.$$
Given $\check{l}\in\mathbb{N}$ and  a  finite subset $\check{J}=\{\check{j}_1<\check{j}_2<\dots \check{j}_{\check{l}}\}$ of $\mathbb{N}$, one has that
\begin{align*}
&\{z\in\boldsymbol{E}: \hat{\tilde{\gamma}}_n(z)=\check{J}\}
\\=&\bigcup_{0\leq j_1<\cdots<j_{\check{l}}}\left(\left\{z\in\boldsymbol{E}: \hat{\underline{\gamma}}_n(z)=\{j_1,j_2,\dots, j_{\check{l}}\}\right\}\cap \{z\in\boldsymbol{E}: \boldsymbol{b}_{j_i}(z)=\check{j}_{i}: i=1,\dots,\check{l}\}\right)
\end{align*}
is a Borel measurable subset of $Y$. It follows  that $\tilde{\gamma}_n$ is Borel measurable.  Note that for any $z\in\boldsymbol{E}$,
\begin{align*}
\hat{\tilde{\gamma}}_n(z)&=\{\boldsymbol{b}_j(z): j\in \hat{\underline{\gamma}}_n(z)\}
\\&=\{\boldsymbol{b}_j(z): j\in J\big(\check{\mathbf{N}}_n(z),\mathcal{R}_{\check{\mathbf{N}}_n(z)}(z)\big)\}
\\&\overset{\eqref{22-12-12-2243}}{\subset} \{\boldsymbol{b}_1(z),\dots, \boldsymbol{b}_{\check{\mathbf{N}}_n(z)}(z)\}.
\end{align*}
For all,  one has that
 \begin{align*}
  |\hat{\tilde{\gamma}}_n(z)|=|\underline{\hat{\gamma}}_n(z)|&\geq \boldsymbol{c}\check{\mathbf{N}}_n(z)
  \\&\overset{\eqref{22-10-19-07}}{=}\boldsymbol{c}|\mathfrak{b}(z)\cap\{0,1,\dots, \tilde{\mathbf{N}}_n(z)-1\}|
  \\&\overset{\eqref{22-12-13-21}}{\geq} \boldsymbol{c}\boldsymbol{c}_1\tilde{\mathbf{N}}_n(z)=:\boldsymbol{b}\tilde{\mathbf{N}}_n(z),
  \end{align*}  and for each $s\in\{1,\dots, \boldsymbol{r}\}^{\hat{\tilde{\gamma}}_n(z)}$
$$\mu_z\left(\pi_1^{-1}z\cap\bigcap_{j\in \hat{\gamma}_n(z)}T^{-j}B_{s(j)}\right)>0.$$
   This completes the proof of  \Cref{22-10-19-05}.
\end{proof}
\textbf{Step 4:}  In this step, we borrow   \Cref{22-10-19-05} to complete the proof of  \Cref{22-12-13-35}. Note that
$$\pi_2: (Y,\mathscr{Y},\nu, S)\to(\bar{\Omega}, \bar{\mathscr{F}},\bar{\mathbb{P}},\bar{\theta})$$
is a factor map between two measure-preserving dynamical systems on Polish probability spaces.  Since $\nu(\boldsymbol{E})=1$ and $\pi_2$ is a Borel measurable map, there exists $\bar{\Omega}_1\in\bar{\mathscr{F}}$ satisfying $\bar{\mathbb{P}}(\bar{\Omega}_1)=1$ and $\pi_2^{-1}(\bar{\omega})\cap \boldsymbol{E}\neq\varnothing$ for each
 $\bar{\omega}\in\bar{\Omega}_1$ (for example, see \cite[Theorem 2.9]{G}). Together with \Cref{22-10-19-06}, there exists a  measurable map $\hat{\pi}_2: \bar{\Omega}\to Y$ and $\mathbb{P}$-full measure $\bar{\Omega}_2\in\mathscr{F}$ such that each $\bar{\omega}\in \bar{\Omega}_2$,
\begin{align}
\hat{\pi}_2(\bar{\omega})\in \boldsymbol{E}\quad\text{ and }\pi_2\circ\hat{\pi}_2(\bar{\omega})=\bar{\omega}.
\end{align}

According to  \Cref{22-10-19-05} in \textbf{Step 3},   there exists a positive constant $\boldsymbol{b}$,  a sequence $\{\mathbf{N}_n\}_{n\in\mathbb{N}}$ of strictly monotone increasing Borel measurable maps $\mathbf{N}_n: \bar{\Omega}_2\to\mathbb{N}$ given by $\tilde{\mathbf{N}}_n\circ\hat{\pi}_2$  and  a sequence $\{\gamma_n\}_{n\in\mathbb{N}}$ of  Borel measurable maps $\gamma_n: \bar{\Omega}_2\to\{0,1\}^{\mathbb{Z}_+}$ given by $\tilde{\gamma}_n\circ\hat{\pi}_2$  such that  for any $\bar{\omega}\in\bar{\Omega}_2$
\begin{enumerate}[(a)]
\item $|\hat{\gamma}_n(\bar{\omega})|\geq \boldsymbol{b}\mathbf{N}_n(\bar{\omega})$;
\item\label{22-10-20-01} for any $s\in\{1, 2\}^{\hat{\gamma}_n(\bar{\omega})}$, one has that
$$\mu_{\hat{\pi}_2(\bar{\omega})}\left(\pi^{-1}_1\Big(\hat{\pi}_2(\bar{\omega})\Big)\cap\bigcap_{j\in \hat{\gamma}_n(\bar{\omega})}T^{-j}B_{s(j)}\right)>0.$$
\end{enumerate}
 It follows that  for any $n\in\mathbb{N}$  and $\bar{\omega}\in\bar{\Omega}_2$ and $s\in \{1, 2\}^{\hat{\gamma}_n(\bar{\omega})}$, there exists
 \begin{align}
 \label{22-12-13-32}
 (\bar{\omega}_s, x_s)\in \pi^{-1}_1\Big(\hat{\pi}_2(\bar{\omega})\Big)\cap \bigcap_{j\in\hat{\gamma}_n(\bar{\omega}) }T^{-j}B_{s(j)}.
 \end{align}
 Noting that  $\pi_2\circ\hat{\pi}_2(\bar{\omega})=\bar{\omega}_s\in\bar{\Omega}$, then $\bar{\omega}=\bar{\omega}_s$. It follows that $\pi_1(\bar{\omega}, x_s)=\hat{\pi}_2(\bar{\omega})\in\boldsymbol{E}$ and    $S^j(\pi_1(\bar{\omega}, x_s))\in \boldsymbol{A}$ for any $j\in\hat{\gamma}_n(\bar{\omega})=\hat{\tilde{\gamma}}_n(\hat{\pi}_2(\bar{\omega}))$. Therefore, 
$$ (\bar{\theta}^j\bar{\omega}, \Phi_{\bar{\omega}}^j(x_s))=T^j(\bar{\omega}, x_s)\in B_{s(j)}\cap \pi_1^{-1}\boldsymbol{A}\subset\bar{\Omega}\times U_{s(j)},$$  
which implies that  $ \Phi_{\bar{\omega}}^j(x_s)\in U_{s(j)}$ for any $j\in\hat{\gamma}_n(\bar{\omega})$.  This finishes the proof of \Cref{22-12-13-35}.
\end{proof}

\section{Proof of  Theorem \ref{23-1-13-0019}}
\label{22-11-03-02}
Under the setting of \Cref{weakhoreshoes},  the collection of  the hitting freedom at each fibers  induces  a product system.  And, we will view the  induced system as a trivial  RDS. By  Krylov-Bogolyubov theorem in RDS, we  prove that GSNS has  full-horseshoes.

\subsection{Krylov–Bogolyubov  theorem in RDSs}

 In this subsection, we mainly review the  narrow topology of probability measures and Krylov–Bogolyubov  theorem on an  invariant  random compact set for a continuous RDS.  For convenience,  we still assume  $M$ is a compact metric space,  $\mathscr{B}_M$ is the Borel $\sigma$-algebra of $M$,   $(\Omega,\mathscr{F},\mathbb{P},\theta)$ is measure-preserving dynamical system on the Polish probability space, 
  and $\mathscr{F}_{\mathbb{P}}$ is the completion of $\mathscr{F}$ with respect to $\mathbb{P}$ through this subsection.

 Denote $C_{\Omega}(M)$  as the collections of  functions $h:\Omega\times M\to \mathbb{R}$ satisfying
 \begin{enumerate}
 \item\label{23-3-8-0014} for all $x\in M$, $\omega\mapsto h(\omega, x)$  is measurable from $(\Omega, \mathscr{F}_{\mathbb{P}})$ to $(\mathbb{R},\mathscr{B}_{\mathbb{R}})$;
 \item\label{23-3-8-0015} for all $\omega\in\Omega$, $x\mapsto h(\omega, x)$ is continuous;
 \item  $\int_{\Omega}\sup_{x\in M}|h(\omega,x)|d\mathbb{P}(\omega)<+\infty$.
 \end{enumerate}
It is pointed out that  if the mapping  $h: \Omega\times M\to \mathbb{R}$ satisfies  \eqref{23-3-8-0014} and \eqref{23-3-8-0015}, then $h $ is measurable from $(\Omega\times M, \mathscr{F}_{\mathbb{P}}\otimes\mathscr{B}_M)$ to $(\mathbb{R},\mathscr{B}_{\mathbb{R}})$ (for example, see \cite[Lemma 1.1]{C}).    Recall that  $\mathcal{P}_{\mathbb{P}}(\Omega\times M)$  is  the space of probability measures on $(\Omega\times M, \mathscr{F}_{\mathbb{P}}\otimes\mathscr{B}_M)$ with the marginal $\mathbb{P}$. The  \emph{narrow topology} of $\mathcal{P}_{\mathbb{P}}(\Omega\times M)$ is defined by the topology  basis which  is given by the collection of all sets of the form
$$U_{h_1,\dots, h_n}(\check{\nu},\delta)=\{\nu\in \mathcal{P}_{\mathbb{P}}(\Omega\times M): |\int_{\Omega\times M}h_k\mathrm{d}\check{\nu}-\int_{\Omega\times M}h_k\mathrm{d}\nu|<\delta, k=1,\dots,n\},$$
where $n\in\mathbb{N}, h_1,\dots,h_n\in C_{\Omega}(M)$, $\check{\nu}\in \mathcal{P}_{\mathbb{P}}(\Omega\times M)$ and $\delta>0$.  

  A  \emph{random compact set} $K$ of  $M$ on the measurable space $(\Omega,\mathscr{F}_{\mathbb{P}})$ is a set-valued map from $\Omega$ to  $2^M$, the collection of all subsets of $M$,  with $\omega\mapsto K(\omega)$ satisfying that
\begin{enumerate}[(i)]
\item\label{22-10-23-01-01} $K(\omega)$ is a non-empty compact subset for any $\omega\in\Omega$;
\item\label{22-10-23-01-02}  for any $x\in M$, $\omega\mapsto d_M(x, K(\omega))$ is a measurable map $(\Omega, \mathscr{F}_{\mathbb{P}})$ to $(\mathbb{R},\mathscr{B}_{\mathbb{R}})$, where $d_M$ is a compatible  metric on $M$.
\end{enumerate}

 Next, we review  two lemmas about  Portmenteau theorem in RDSs and the equivalent characterization of a random compact set, respectively.

 \begin{lemma}[{\cite[Theorem 3.17]{C}}]
 \label{22-12-14-309}
 Let $\{\tilde{\nu}_n\}_{n\in\mathbb{N}}$ be a sequence of  $\mathcal{P}_{\mathbb{P}}(\Omega\times M)$. Then   $\{\tilde{\nu}_n\}_{n\in\mathbb{N}}$ converges to  a measure $\tilde{\nu}\in \mathcal{P}_{\mathbb{P}}(\Omega\times M) $ in the narrow topology if and only if  
 $$\limsup_{n\to+\infty}\tilde{\nu}_n(K)\leq\tilde{\nu}(K)$$
  for any random compact set $K$ of $M$ on $(\Omega,\mathscr{F}_{\mathbb{P}})$.
 \end{lemma}
  
\begin{lemma}[{\cite[Proposition 1.6.2 and Proposition 1.6.3]{A}}]
\label{22-11-1-01}
 Let $K: \Omega\to 2^M$ be a set-valued map taking values in the subspace of non-empty compact subsets of  $M$.  Then the following statements are equivalent:
 \begin{enumerate}[(a)]
  \item $ K\in\mathscr{F}_{\mathbb{P}}\otimes \mathscr{B}_M$;
  \item $K$ is a random compact set of $M$ on $(\Omega,\mathscr{F}_{\mathbb{P}})$;
  \end{enumerate} 
\end{lemma}
  Let $F$ be a continuous RDS on $M$ over  $(\Omega,\mathscr{F}_{\mathbb{P}},\mathbb{P},\theta)$.   A random compact set $K$ of $M$ on $(\Omega,\mathscr{F}_{\mathbb{P}})$ is said to be $F$-forward invariant if  $F^n_{\omega}\big(K(\omega)\big)\subset K(\theta^n\omega)$ for  any $n\in\mathbb{N}$  and  $\omega\in\Omega$.  According to \cite[Corollary 6.13]{C}, one has the Krylov–Bogolyubov  theorem:
 \begin{lemma}
\label{22-10-12-01}
 Let $F$ be a continuous RDS on $M$ over  $(\Omega,\mathscr{F}_{\mathbb{P}},\mathbb{P},\theta)$ and   $K$ be a $F$-forward invariant random compact  set of $M$ on $(\Omega,\mathscr{F}_{\mathbb{P}})$.  Assume that $\{\gamma_n\}_{n\in\mathbb{N}}$ is a sequence of measurable maps $\gamma_n: (\Omega,\mathscr{F}_{\mathbb{P}})\to (M,\mathscr{B}_M)$ satisfying $\gamma_n(\omega)\in K(\omega)$ for $\mathbb{P}$-a.s. $\omega\in\Omega$,  and $\{\mathbf{N}_n\}_{n\in\mathbb{N}}$  is a sequence of  strictly increasing measurable maps $\mathbf{N}_n:(\Omega,\mathscr{F}_{\mathbb{P}})\to(\mathbb{Z}_+, \mathscr{B}_{\mathbb{Z}_+})$.   For a sequence of probability measures $\{\tilde{\nu}_n\}_{n\in\mathbb{N}}$ on $(\Omega\times M,\mathscr{F}_{\mathbb{P}}\otimes\mathscr{B}_M)$  defined as 
$$\tilde{\nu}_n=\int_{\Omega}\frac{1}{\mathbf{N}_n(\omega)}\sum_{i=0}^{\mathbf{N}_n(\omega)-1}\delta_{\big(\theta^i\omega,F^i_{\omega}\gamma_n(\omega)\big)}\mathrm{d}\mathbb{P}(\omega),$$
there exists a strictly increasing sequence $\{n_k\}_{k\in\mathbb{N}}$ of $\mathbb{N}$  such that $\tilde{\nu}:=\lim_{k\to+\infty}\tilde{\nu}_{n_k}$ is an invariant measure of the RDS $F$,  and $\tilde{\nu}(K)=1$.
\end{lemma}

 \subsection{Proof of  Theorem \ref{23-1-13-0019}}
\label{22-12-14-126}
In this subsection, we give the final proof of  Theorem \ref{23-1-13-0019}.    Recall that the  one to one corresponding relationship between the element in $\{0,1\}^{\mathbb{Z}_+}$ and subset of $\mathbb{Z}_+$ is  defined as 
 \begin{align*}
 u\in\{0,1\}^{\mathbb{Z}_+}\mapsto \hat{u}=\{n\in\mathbb{Z}_+: u(n)=1\}.
 \end{align*}
Now,  we give a general result to  guarantee the existence of full-horseshoes on two disjoint closed balls.   \begin{thm}
\label{23-10-10-2137}
If stochastic flow $\Phi$ on $\mathbb{R}^d$ over $(\Omega,\mathscr{F},\mathbb{P})$ defined as \eqref{23-2-18-1254}   has measurable weak-horseshoes on two disjoint closed balls, then it has full-horseshoes.
\end{thm}
 \begin{proof}
 According to the definition of measurable weak-horseshoes,  we know that there exists a pair of   disjoint closed balls  $\{U_1, U_2\}$   of $\mathbb{R}^d$, a constant $\boldsymbol{b}>0$, a $\mathbb{P}$ full-measure $\Omega^1\subset\Omega$,  a sequence $\{\mathbf{N}_n\}_{n\in\mathbb{N}}$ of strictly  increasing Borel measurable map $\mathbf{N}_n:\Omega\to\mathbb{Z}_+$,  and a sequence $\{\gamma_n\}_{n\in\mathbb{N}}$ of Borel measurable maps $\gamma_n: \Omega\to \{0,1\}^{\mathbb{Z}_+}$
such that one has that for any $n\in\mathbb{N}$ and any $\omega\in\Omega^1$, 
\begin{enumerate}[(a)]
\item $\hat{\gamma}_n(\omega)\subset\{0,1,\dots, \mathbf{N}_n(\omega)-1\}$ and  $|\hat{\gamma}_n(\omega)|\geq \boldsymbol{b}\mathbf{N}_n(\omega)$;
\item  for any $s\in\{1,2\}^{\hat{\gamma}_n(\omega)}$, there exists an $x_s\in \mathbb{R}^d$ with $ \Phi^j_{\omega}(x_s)\in U_{s(j)}$ for any $j\in \hat{\gamma}_n(\omega)$.
\end{enumerate}

Recall that  $\theta^1:\Omega\to\Omega$ is the Wiener shift  defined as $\theta^1(\omega)=\omega(\cdot+1)-\omega(1)$ and $(\Omega,\mathscr{F}_{\mathbb{P}},\mathbb{P},\theta^1)$  is an ergodic measure-preserving dynamical system,  where $\mathscr{F}_{\mathbb{P}}$ is the completion of $\mathscr{F}$ with respect to $\mathbb{P}$.  Define a continuous RDS $F$ on $\{0,1\}^{\mathbb{Z}_+}$ over $(\Omega,\mathscr{F}_{\mathbb{P}},\mathbb{P},\theta^1)$ by setting $F_{\omega}=\sigma$ as  the left-shift on
$\{0,1\}^{\mathbb{Z}_+}$, i.e.
$$F: \mathbb{Z}_+\times \Omega\times \{0,1\}^{\mathbb{Z}_+}\to \{0,1\}^{\mathbb{Z}_+},\quad (n,\omega,u)\mapsto \sigma^n(u).$$
For any $u\in\{0,1\}^{\mathbb{Z}_+}$ and  $\omega\in\Omega$,  denote
\begin{align*}
K(\omega):=\Big\{ u\in\{0,1\}^{\mathbb{Z}_+}:&\text{ for any } s\in \{1,2\}^{\hat{u}}, \text{ there exists } x_s\in\mathbb{R}^d
\\&\text{ such that } \Phi^j_{\omega}(x_s)\in U_{s(j)} \text{ for  each } j\in \hat{u}\Big\}.
\end{align*}
 Here, we need to point out that  we regard $u=\boldsymbol{0}=(0,0,\cdots)\in\{0,1\}^{\mathbb{Z}_+}$ as an element of $K(\omega)$ for any $\omega\in\Omega$. Denote $K:=\bigcup_{\omega\in\Omega}\{\omega\}\times K(\omega)$. In the next, we prove that the slight adjustment of $K$ is a $F$-forward invariant random compact set. Namely, 
\begin{lemma}
\label{22-11-14-02}
There exists  a $F$-forward invariant random compact set $\tilde{K}$ of $\{0,1\}^{\mathbb{Z}_+}$ on $(\Omega,\mathscr{F}_{\mathbb{P}})$ such that $\tilde{K}(\omega)=K(\omega)$ for $\mathbb{P}$-a.s. $\omega\in\Omega$.
\end{lemma}
\begin{proof}
Fix $(\omega,u)\in K$.  If $u=\boldsymbol{0}$, then it is clear that $\sigma(u)=\boldsymbol{0}\in K(\theta^1\omega)$. If $u\neq\boldsymbol{0}$, 
note that   $\{j+1: j\in \hat{\sigma(u)}\}\subset \hat{u}$. For any $s\in \{1, 2\}^{\hat{\sigma(u)}}$,  there exists  $\tilde{s}\in\{1,2\}^{\hat{u}}$ such that $\tilde{s}(j+1)=s(j)$ for each $j\in \hat{\sigma(u)}$.  Since $u\in K(\omega)$,  we can find $x_{\tilde{s}}\in \mathbb{R}^d$ such that 
  $$\Phi^{j+1}_{\omega}(x_{\tilde{s}})\in U_{\tilde{s}(j+1)}\quad \text{for each } j\in \hat{\sigma(u)}.$$ 
    Then, 
$$\Phi^j_{\theta\omega}\big(\Phi^1_{\omega}x_{\tilde{s}}\big)=\Phi^{j+1}_{\omega}( x_{\tilde{s}})\in U_{\tilde{s}(j+1)}=U_{s(j)}\quad\text{for each } j\in \hat{\sigma(u)}.$$
Therefore, $(\theta\omega,\sigma(u))\in K$ which implies that $K$ is $F$-forward invariant.    The remainder proof of this lemma is divided into two steps.

\textbf{Step 1}: Let $\Omega^2$ be a $\mathbb{P}$-full measure  subset of $\Omega^1$ such that for any $l, m\in\mathbb{Z}_+$,  the mapping $\omega\mapsto \Phi_{\theta^l\omega}^m$  from $\Omega^2$ to $\text{Diff}^{\infty}(\mathbb{R}^d)$  is a Borel measurable map. 
  By  Lusin's theorem (for example, see \cite[(17.12) Theorem]{MR1321597}), there exists a sequence of compact subsets $\{\Omega_n\}_{n\in\mathbb{N}}$ of $\Omega$ with $\mathbb{P}(\Omega_n)\geq 1-1/n$ such that $\Omega_n\subset\Omega^2$ and 
  \begin{align}
\label{22-10-05-01}
&\Omega_n\times \mathbb{R}^d\to \mathbb{R}^d,\quad (\omega,x)\mapsto \Phi_{\theta^l\omega}^m(x)
\end{align}
is continuous for any $l,m \in\mathbb{N}$.   In this step, we show  that for each $n\in\mathbb{N}$,
\begin{align*}
K_{n}=\bigcup_{\omega\in\Omega_n}\{\omega\}\times K(\omega)= K\cap (\Omega_n\times\{0,1\}^{\mathbb{Z}_+})
\end{align*} is a closed subset of $\Omega\times\{0,1\}^{\mathbb{Z}_+}$. Moreover,  for any  $\omega\in \Omega_n$, $K(\omega)$ is a compact  subset of $\{0,1\}^{\mathbb{Z}_+}$.

Given a sequence $\{(\omega_i ,u_i)\}_{i\in\mathbb{N}}$ of $K_{n}$ satisfying
$$(\omega, u):=\lim_{i\to+\infty}(\omega_i, u_i)\in\Omega\times\{0,1\}^{\mathbb{Z}_+},$$
we are going to show that $(\omega,u)\in K_n$.  Since $\omega=\lim_{i\to+\infty}\omega_i\in\Omega_n$, we only need to prove that $u\in K(\omega)$. If $\hat{u}=\varnothing$, then  $u= \mathbf{0}$ and $(\omega, \mathbf{0})\in K_n$.  Otherwise,  if $\hat{u}\neq\varnothing$,  denote $n_u=\min\{n\in\mathbb{Z}_+: n\in\hat{u}\}$.  Fixing $\check{s}\in\{1,2\}^{\hat{u}}$,  there exists  a strictly increasing sequence $\{i_k\}_{k\in\mathbb{N}}$ of $\mathbb{N}$ such that for any $k\in\mathbb{N}$,
$$\hat{u}_{i_k}\cap\{0,1,\dots, n_u+r\}=\hat{u}\cap\{0,1\dots, n_u+r\},$$
where $1\leq r\leq k$.   Now for each $k\in\mathbb{N}$,  we can choose a  $s^k\in\{1,2\}^{\hat{u}_{i_k}}$ such that 
\begin{align}
\label{22-12-14-222}
s^k(j)=\check{s}(j)\text{ for }j\in\hat{u}\cap\{0,1,\dots, n_u+k\}.
\end{align}
Since $(\omega_{i_k},u_{i_k})\in K_{n}$, there exists $x_{s^k}\in \mathbb{R}^d$ such that
$\Phi^j_{\omega_{i_k}}(x_{s^k})\in U_{s^k(j)}$  for each $j\in\hat{u}_{i_k}$.  By   compactness of  $U_{\check{s}(n_u)}$ ,  without loss of generality, we assume that
\begin{align}
\label{22-10-23-06}
\lim_{k\to+\infty} \Phi^{n_u}_{\omega_{i_k}}(x_{s^k})=x^{n_u}_s.
\end{align}
Since that  $\Phi^{n_u}_{\omega}$ is a diffeomorphism on $\mathbb{R}^d$,  there is an $x_s\in \mathbb{R}^d$ such that  $\Phi^{n_u}_{\omega}(x_s)=x_s^{n_u}$.  Therefore,    for each $j\in\hat{u}$,
 \begin{align*}
 \Phi^j_{\omega}(x_s)&=\Phi^{j-n_u}_{\theta^{n_u}\omega}(x_s^{n_u})
 \\&=\lim_{k\to+\infty}\Phi^{j-n_u}_{\theta^{n_u}\omega_{i_k}}\big(\Phi^{n_u}_{\omega_{i_k}}(x_{s^k})\big)
 \\&=\lim_{k\to+\infty}\Phi^{j}_{\omega_{i_k}}\big(x_{s^k}\big)\in U_{s(j)},
 \end{align*}
which implies that $(\omega,u)\in K_n$. Hence that, $K_n$ is a closed subsets of $\{0,1\}^{\mathbb{Z}_+}$.
 
\textbf{Step 2}: In this step,  we prove that a slight adjustment of $K$ is a measurable subset of $\Omega\times\{0,1\}^{\mathbb{Z}_+}$. This  provides  the existence of  $\tilde{K}$. 

 Denoting  $\tilde{\Omega}:=\bigcup_{n\in\mathbb{N}}\Omega_n$, then $\Omega^3:=\bigcap_{j\in\mathbb{Z}_+}(\theta^{j})^{-1}\tilde{\Omega}$ is a $\theta^1$-forward invariant $\mathbb{P}$-full measure subset of $\Omega^2$. Let
$$\tilde{K}=\big(K\cap(\Omega^3\times\{0,1\}^{\mathbb{Z}_+})\big)\cup\big( (\Omega\setminus\Omega^3)\times\{\mathbf{0}\}\big),$$
 Note that 
\begin{itemize}
\item  for any $n\in\mathbb{N}$ and $\omega\in\Omega$, $F^n_{\omega}(\tilde{K}(\omega))=\sigma^n(\tilde{K}(\omega))\subset\tilde{K}(\theta^n\omega)$. Hence that $\tilde{K}$ is   $F$-forward invariant;
\item  $K(\omega)$ is a compact subset of $\{0,1\}^{\mathbb{Z}_+}$ for any $\omega\in\Omega^3$; 
\item since $K_n$ is a closed subset of $\Omega\times\{0,1\}^{\mathbb{Z}_+}$ for any $n\in\mathbb{N}$, one has that 
\begin{align*}
K\cap (\Omega^3\times \{0,1\}^{\mathbb{Z}_+})&=K\cap (\tilde{\Omega}\times   \{0,1\}^{\mathbb{Z}_+})\cap(\Omega^3\times  \{0,1\}^{\mathbb{Z}_+})
\\&=K\cap\big(\bigcup_{n\in\mathbb{N}}(\Omega_n\times \{0,1\}^{\mathbb{Z}_+})\big)\cap(\Omega^3\times  \{0,1\}^{\mathbb{Z}_+})
\\&=\bigcup_{n\in\mathbb{N}}(K\cap(\Omega_n\times \{0,1\}^{\mathbb{Z}_+}))\cap(\Omega^3\times  \{0,1\}^{\mathbb{Z}_+})
\\&=\bigcup_{n\in\mathbb{N}} K_n\cap(\Omega^3\times  \{0,1\}^{\mathbb{Z}_+})\in\mathscr{F}_{\mathbb{P}}\otimes\mathscr{B}_{\{0,1\}^{\mathbb{Z}_+}}.
\end{align*}
 \end{itemize}
  This finishes the proof of \Cref{22-11-14-02} by using \Cref{22-11-1-01}.
\end{proof}

Let's proceed to prove  \Cref{23-10-10-2137}.  Define a sequence of probability measures  $\{\tilde{\nu}_n\}_{n\in\mathbb{N}}$ on $(\Omega\times\{0,1\}^{\mathbb{Z}_+},\mathscr{F}_{\mathbb{P}}\otimes\mathscr{B}_{\{0,1\}^{\mathbb{Z}_+}})$  as follows,
\begin{align}
\tilde{\nu}_n=\int_{\Omega}\frac{1}{\mathbf{N}_n(\omega)}\sum_{i=0}^{\mathbf{N}_n(\omega)-1}\delta_{\big(\theta^i\omega, \sigma^i\gamma_n(\omega)\big)}\mathrm{d}\mathbb{P}(\omega),
\end{align}
where $\{\gamma_n\}_{n\in\mathbb{N}}$ is a sequence of Borel measurable maps $\gamma_n: \Omega\to\{0,1\}^{\mathbb{Z}_+}$ and $\{\mathbf{N}_n\}_{n\in\mathbb{N}}$  is a sequence of Borel measurable maps  $\mathbf{N}_n: \Omega\to\mathbb{N}$, which  are given in the beginning of the proof of \Cref{23-10-10-2137}.   It is clear that for any $n\in\mathbb{N}$ and $\mathbb{P}$-a.s. $\omega\in\Omega$, one has that $\gamma_n(\omega)\in\tilde{K}(\omega)$. By \Cref{22-10-12-01},   there exists  a strictly increasing sequence $\{n_k\}_{k\in\mathbb{N}}$  of  $\mathbb{N}$ such that $\tilde{\nu}:=\lim_{k\to+\infty}\tilde{\nu}_{n_k}$ is an invariant measure of the RDS $F$  and $\tilde{\nu}(\tilde{K})=1$.   By  \Cref{22-12-14-309}, one has that
 \begin{align}
 \lim_{k\to+\infty}\tilde{\nu}_{n_k}([1])=\tilde{\nu}([1]),
 \end{align}
 where we used that $[1]:=\Omega\times\{u\in\{0,1\}^{\mathbb{Z}_+}: u(0)=1\}$ and 
 $$[0]:=\big(\Omega\times\{0,1\}^{\mathbb{Z}_+}\big)\setminus[1]=\Omega\times\{u\in\{0,1\}^{\mathbb{Z}_+}: u(0)=0\}$$ both are random compact set of $\{0,1\}^{\mathbb{Z}_+}$ on $(\Omega,\mathscr{F}_{\mathbb{P}})$. Therefore,
\begin{align*}
\tilde{\nu}([1])&=\lim_{k\to+\infty}\tilde{\nu}_{n_k}([1])
\\&=\lim_{k\to+\infty}\int_{\Omega}\frac{1}{\mathbf{N}_{n_k}(\omega)}\sum_{j=0}^{\mathbf{N}_{n_k}(\omega)-1}\delta_{\big(\theta^j\omega, \sigma^j\gamma_{n_k}(\omega)\big)}([1])\mathrm{d}\mathbb{P}(\omega)
\\&=\lim_{k\to+\infty}\int_{\Omega}\frac{|\{j\in\{0,1,\dots, \mathbf{N}_{n_k}(\omega)-1\}: \big(\sigma^j\gamma_{n_k}(\omega)\big)(0)=1\}|}{\mathbf{N}_{n_k}(\omega)}\mathrm{d}\mathbb{P}(\omega)
\\&=\lim_{k\to+\infty}\int_{\Omega}\frac{|\{j\in\{0,1,\dots, \mathbf{N}_{n_k}(\omega)-1\}: \big(\gamma_{n_k}(\omega)\big)(j)=1\}|}{\mathbf{N}_{n_k}(\omega)}\mathrm{d}\mathbb{P}(\omega)
\\&=\lim_{k\to+\infty}\int_{\Omega}\frac{|\hat{\gamma}_{n_k}(\omega)|}{\mathbf{N}_{n_k}(\omega)}\mathrm{d}\mathbb{P}(\omega)\geq \boldsymbol{b}.
\end{align*}
By the ergodic decomposition (for example, see \cite[Theorem 6.2]{EW}), $\pi_*\tilde{\nu}=\mathbb{P}$ and the fact that $(\Omega,\mathscr{F}_{\mathbb{P}},\mathbb{P},\theta)$ is ergodic, we know that there exists an invariant ergodic Borel probability measure $\nu$ of the RDS $F$  such that  $\nu([1])=\tilde{\nu}([1])\geq\boldsymbol{b}$ and $\nu(\tilde{K})=1$.

Let 
$$G_{\nu}:=\left\{(\omega,u)\in\Omega\times\{0,1\}^{\mathbb{Z}_+}: \lim_{N\to+\infty}\frac{1}{N}\sum_{j=0}^{N-1}\delta_{(\theta^j\omega, \sigma^ju)}([1])=\nu([1])\right\}.$$  By Birkhoff ergodic theorem, one has that $\nu(G_v)=1$.   Then there exists a  $\mathbb{P}$-full measure subset $\Omega^4$ of $\Omega^3$  such that for any $\omega\in\Omega^4$
$$\pi_{\Omega}^{-1}(\omega)\cap G_v\cap K=\pi_{\Omega}^{-1}(\omega)\cap G_v\cap\tilde{K}\neq\varnothing,$$ where $\pi_{\Omega}$ is the projection form $\Omega\times \{0,1\}^{\mathbb{Z}_+}$ to $\Omega$.
For any $\omega\in\Omega^4$, there exists $u_{\omega}\in \{0,1\}^{\mathbb{Z}_+}$ such that $(\omega, u_{\omega})\in G_{\nu}\cap K$.
Letting $J(\omega):=\hat{u}(\omega)=\{n\in\mathbb{Z}_+: u_{\omega}(n)=1\}$,  then one has that
\begin{align*}
\nu([1])&=\lim_{N\to+\infty}\frac{1}{N}\sum_{j=0}^{N-1}\delta_{(\theta^j\omega,\sigma^ju_{\omega})}([1])
\\&=\lim_{N\to+\infty}\frac{1}{N}|\{j\in\{0,1,\dots, N-1\}: (\sigma^ju_{\omega})(0)=1\}|
\\&=\lim_{N\to+\infty}\frac{1}{N}|\{j\in\{0,1,\dots, N-1\}: u_{\omega}(j)=1\}|
\\&=\lim_{N\to+\infty}\frac{1}{N}|J(\omega)\cap\{0,1,\dots, N-1\}|\geq\boldsymbol{b}.
\end{align*}
By the definition of $K$,  for any $s\in\{1,2\}^{J(\omega)}$, there exists an $x_s\in \mathbb{R}^d$ with $ \Phi^j_{\omega}(x_s)\in U_{s(j)}$ for any $j\in J(\omega)$. For all, we complete the proof of \Cref{23-1-13-0019}.
\end{proof}
 Therefore, Theorem \ref{23-1-13-0019} follows from \Cref{weakhoreshoes} and \Cref{23-10-10-2137}. 
Through  all proofs,  it can  be seen that,
\begin{proposition}
\label{23-3-13-2003}
For any a stochastic flow $\Phi$ of $C^2$ diffeomorphisms  on $\mathbb{R}^d$ over the Wiener space $(\Omega,\mathscr{F},\mathbb{P})$  which is defined as \eqref{23-2-18-1254},  if it satisfies following hypothesis, 
\begin{itemize}
\item[(H1)] $\Phi$ admits a smooth  stationary measure which is ergodic with respect to the its  time-1 transition probabilities and  has positive top Lyapunov exponents with respect to  this stationary measure;
\item[(H2)] \textbf{Assumption 1-Assumption 3} in \Cref{23-2-16-1653}  holds,
\end{itemize}
then the stochastic flow $\Phi$ has full-horseshoes.
\end{proposition}

\providecommand{\bysame}{\leavevmode\hbox to3em{\hrulefill}\thinspace}
\providecommand{\MR}{\relax\ifhmode\unskip\space\fi MR }
\providecommand{\MRhref}[2]{%
  \href{http://www.ams.org/mathscinet-getitem?mr=#1}{#2}
}
\providecommand{\href}[2]{#2}

\end{document}